\documentclass{amsart}
\usepackage{amsmath, amssymb, 
graphicx, mathrsfs, enumerate}
\usepackage[all]{xy}
\usepackage{xcolor}
\usepackage{latexsym}

\newtheorem{prop}{Proposition}[section]
\newtheorem{lemma}[prop]{Lemma}
\newtheorem{corollary}[prop]{Corollary}
\newtheorem{theorem}[prop]{Theorem}
\newtheorem*{theorem_A}{Theorem A}
\newtheorem*{theorem_B}{Theorem B}
\newtheorem{definition}[prop]{Definition}
\newtheorem{remark}[prop]{Remark}
\newtheorem{example}{Example}
\newtheorem{problem}{Open Problem}

\newtheorem{conjecture}{Conjecture}
\newcommand{\cupdot}{\mathbin{\mathaccent\cdot\cup}}
    \def\PSL{{\rm PSL}}
        \def\PGL{{\rm PGL}}
        \def\PSU{{\rm PSU}}
        \def\Sz{{\rm Sz}}
        \def\SL{{\rm SL}}
        \def\SU{{\rm SU}}
        \def\Sp{{\rm Sp}}
        \def\PSp{{\rm PSp}}
        \def\POm{{\rm P}\Omega}
        \def\Diag{{\rm Diag}}
        \def\GL{{\rm GL}}
        \def\F{\mathbb{F}}
        \def\PP{\mathcal{P}}
        \def\Cyc{{\rm Cyc}}
\newcommand{\set}[1]{\left\{#1\right\}}               %  insieme
\newcommand{\seq}[1]{\left<#1\right>}                 %  sottogr generato     

\newcommand{\D}[1]{\textcolor{red}{#1}}
\newcommand{\Fr}[1]{\textcolor{purple}{#1}}
\newcommand{\coloneq}{\mathrel{\mathop:}=}

\begin{document}
\title[Enhanced power graphs]{Enhanced power graphs of finite groups with cograph structure}

\author{Daniela Bubboloni} 
 \address{Dipartimento di Matematica  e
      Informatica 'Ulisse Dini', 
		Viale Morgagni 67/A, 50134 Firenze, 
      Italy}
 \email{daniela.bubboloni@unifi.it}

\author{Francesco Fumagalli} 
 \address{Dipartimento di Matematica  e
      Informatica 'Ulisse Dini', 
		Viale Morgagni 67/A, 50134 Firenze, 
      Italy}
	\email{francesco.fumagalli@unifi.it}
	
\author{Cheryl E. Praeger}
 \address{Centre for the Mathematics of Symmetry and Computation, University of Western Australia, Perth 6009, Australia}
 \email{Cheryl.Praeger@uwa.edu.au}

\thanks{The first two authors are partially supported by the Italian INdAM-GNSAGA. The third author acknowledges the hospitality of DIMAI and financial support from GNSAGA which led to this  fruitful collaboration.
This work is also funded by the national project PRIN 2022- 2022PSTWLB - Group Theory and Applications - CUP B53D23009410006. 
}

\begin{abstract} 
The enhanced power graph, 
$\mathcal{E}(G)$, of a group 
$G$ has vertex set $G$ 
and two elements are adjacent 
if they generate a cyclic 
subgroup. In the case of  
finite groups, we 
identify some striking and 
unexpected properties of 
these graphs,  as well as 
links between properties of 
$\mathcal{E}(G)$ and 
properties of the group $G$. 
We prove that if 
$\mathcal{E}(G)$ is a cograph 
then it is also a chordal 
graph. Making use of 
properties of simplicial 
vertices, we characterise the 
finite groups $G$ whose 
enhanced power graph is 
diamond-free or a block 
graph. We also characterise 
the finite groups having 
enhanced power graph a 
cograph or a quasi-threshold 
graph, and those with 
$C_4$-free enhanced power 
graph. We use these 
characterisations to classify 
the finite nonabelian simple groups 
whose enhanced power graph is 
a cograph and give 
information on the finite 
simple groups whose enhanced 
power graph is $C_4$-free. 
Some open problems are posed.

\bigskip
\emph{Dedicated to the memory of our friend and colleague Carlo Casolo}
\end{abstract}
\maketitle

\section{Introduction and main results}\label{s:intro}

There are many graphs associated with groups, each focusing on some specific structural property of the group, see, for instance, \cite{alip, 1955Commuting, britnell,  BParxiv, Bubboloni_3, Bubboloni_4, lucchini} and the comprehensive survey in \cite{GraphsOnGroups}.
Typically, a graph $\Gamma(G)$
associated with a group $G$ reduces the complexity of the algebraic structure of $G$ to a more manageable combinatorial object $\Gamma(G)$, yielding information about $G$. Sometimes this information is scarcely conceivable when looking directly at $G$,  and yet appears completely natural when considering $\Gamma(G)$.
One remarkable example  
is given by the \textbf{enhanced power graph} $\mathcal{E}(G)$, introduced in \cite{alip} which has vertex set $G$ with distinct vertices adjacent if and only if
they generate a cyclic subgroup. The graph theoretic properties of $\mathcal{E}(G)$ sheds  light, often unexpectedly, on how  the cyclic subgroups of $G$ interact with each other. 
For instance,  P. J. Cameron and V. Phan \cite[Theorem 1]{CAPH} recently showed that, for every group $G$ which is finite or infinite with bounded exponent, the clique number and chromatic number of  $\mathcal{E}(G)$ are equal, that is $\mathcal{E}(G)$ is weakly perfect, and that number is  equal to the maximum order of an element of $G$. 
On the other hand, for some classic graph theoretic  property, we might ask for the groups $G$ for which $\mathcal{E}(G)$ has the property, and discover an interesting family of groups for which the cyclic subgroups produce some peculiar configurations. 
The properties we are dealing with in this paper, are to a large extent those concerning forbidden subgraphs: for a given family $\mathcal{X}$ of graphs, we say that a graph $\Gamma$ is \textbf{$\mathcal{X}$-free} if no induced subgraph of $\Gamma$ is isomorphic to a graph in the family $\mathcal{X}$.

Motivated by an open question posed by Cameron in \cite[Question 14]{GraphsOnGroups}  we consider,  for the enhanced power graph, the property of being a \textbf{cograph}, that is, of being $P_4$-free, where $P_4$ denotes a path with four vertices.  
Cographs are a class of graphs with deep structural and algorithmic significance, stemming from their multiple characterisations and beautiful properties (see \cite[Section 2.1]{alecu} and \cite[Theorem 2, Lemmata 3-5]{corneil}).  For example, since $P_4$ is self-complementary (that is, the complement of $P_4$ is isomorphic to $P_4$), the class of cographs is closed under complementation, a rare and powerful property. 
From the computational point of view,  a cograph can be recognised in linear time (see, for example, \cite{BCHP}) and its structure, especially its so-called cotree structure, allows problems such as isomorphism testing, finding the chromatic number or clique number or independence number, and determining whether the cograph has a Hamiltonian path or cycle, to be solved in linear time (see \cite{ JO}). 
It is also worth noting that cographs find application in representation problems in algebraic and logic contexts (see, for example, \cite{albert} and \cite{golumbic}).

For $n\ge3$, we denote by $C_n$ a cycle graph with $n$ vertices. For graphs in general, there is no relationship between the property of being $P_4$-free and that of being $C_4$-free. %\Fr{ Indeed $P_4$ is $C_4$-free but not $P_4$-free; $C_4$ is $P_4$-free but not $C_4$-free. [I would delete this].} 
Surprisingly, the situation is very different if we consider only enhanced power graphs of groups: namely, if an enhanced power graph $\mathcal{E}(G)$ is $P_4$-free, then $\mathcal{E}(G)$ is also $C_4$-free. Further, as an easy consequence of this, if $\mathcal{E}(G)$ is $P_4$-free, then $\mathcal{E}(G)$ is also
$C_n$-free for every $n\geq 4$, that is,
it is a \textbf{chordal graph}.
This implication is  the first fundamental result of our paper.

\begin{theorem_A} Let $G$ be a finite group such that $\mathcal{E}(G)$  is a cograph. Then $\mathcal{E}(G)$  is chordal.
\end{theorem_A}

We emphasise that this fact, though strongly suggested by some results about chordal/cograph enhanced power graphs  in the literature, such as those in \cite{Ma24}, has to our knowledge not been observed previously.

Chordal graphs are a fundamental class of graphs due to their elegant structural properties and efficient algorithmic tractability. Many difficult  problems in graph theory are solved easily for chordal graphs. For example, a chordal graph admits a perfect elimination ordering, where vertices can be ordered such that, for each vertex, its neighbours that appear later in the order form a clique. This property  
makes chordal graphs, along with cographs,  an important and  extensively studied subclass of perfect graphs (graphs for which the chromatic number equals the clique number for every induced subgraph), see \cite{golumbic2}. 

A major difference between cographs and chordal graphs is the fact that the complement of a chordal graph is not necessarily chordal.
This property makes the implication in  
Theorem A  a somewhat unexpected result. %fact, because
Indeed, the 
class of graphs which are both cographs 
and chordal is  quite 
restricted. In the graph theoretic literature it is called the class of \textbf{quasi-threshold graphs} and contains the class of  \textbf{threshold 
graphs} (the $\{P_4, C_4,2K_2\}$-free graphs), see Section \ref{s:threshold}. By \cite[Theorem 1]{Ma24}, the class of finite groups 
whose enhanced power graph is a threshold graph 
 consists only of 
cyclic groups, dihedral groups and 
elementary abelian $2$-groups. On the other hand, the class of finite groups 
whose enhanced power graph is a quasi-threshold graph is, by Theorem A, 
the class of finite groups whose enhanced power graph is a cograph, and it contains infinitely many finite nonabelian simple groups. 
For example, it was shown by Cameron~\cite[Proposition 8.7]{GraphsOnGroups} that $\mathcal{E}(G)$ is a cograph for each of the groups $G=\PSL_2(q)$ with $q\geq4$. We obtain a new proof of this result of Cameron as a corollary of a more general result for groups admitting a partition consisting of cyclic subgroups (Theorem \ref{p:partition}).
%by the fact that for enhanced power graph being a cograph is equivalent to be quasi-threshold (Corollary \ref{thresh}), 
Our second major result includes a classification of all finite nonabelian simple groups whose enhanced power graph is a quasi-threshold graph, or equivalently a cograph. There are precisely two infinite families and one individual example.

\begin{theorem_B}%\label{t:main}
Let $G$ be a finite nonabelian simple group. Then the following are equivalent.
\begin{itemize}
    
\item[$(1)$] $\mathcal{E}(G)$ is a cograph;
\item[$(2)$] $G$ lies in the list \   
$\mathcal{L}_1=\{ \PSL_2(q),\PSL_3(4), {\rm Sz}(q)\}$, where $q\geq4$ for $\PSL_2(q)$ and $q=2^{2m+1}\geq 8$ for ${\rm Sz}(q)$;
%\D{where $q$ can assume any value $q=p^f$, with $p $ prime and $f\geq 1$ in the projective linear case; any value $q=2^{2m+1}$, for some 
%$m\geq 1$ in the Suzuki case; any value $q=3^n$ with  $n=2m+1>1$  in the Ree case.}\\

\item[$(3)$] each pair of distinct maximal cyclic subgroups of $G$  intersects in a $2$-group.
%Every intersection of pairwise distinct maximal cyclic subgroups of $G$ is a $2$-subgroup of $G$.
\end{itemize}
Moreover, the set $\mathcal{S}$ of  finite nonabelian simple groups $G$ such that $\mathcal{E}(G)$ is $C_4$-free satisfies %then $G$ lies in the following list: 
$$\mathcal{L}_1 \cup \set{A_7, M_{11}, J_1}\subseteq \mathcal{S} \subseteq \mathcal{L}_1\cup\set{A_7,
M_{11},  J_1, \PSL_3(q), \PSU_3(q),  {\rm Ree}(q),  {}^2F_4(2)' }
$$
where $q$ is an odd prime  and $(q-1)/(\gcd(3,q-1)$ is a prime power for $\PSL_3(q)$,  $q>2$ for $\PSU_3(q)$, and $q=3^{2m+1}>3$ for ${\rm Ree}(q)$. 
%\end{namedtheorem}
\end{theorem_B}
Central to our proof of Theorem B are our  characterisations, in terms of subgroup properties, of finite groups $G$ for which $\mathcal{E}(G)$ is a cograph (Proposition~\ref{p:W}) or is $C_4$-free (Proposition~\ref{p:AB}).
Also we note that our proof of Theorem B relies on the classification of finite nonabelian simple groups (CFSG). It would be interesting to have a precise description  of the set $\mathcal{S}$ in Theorem~B.

\begin{problem}\label{prob1}
    For a given $n\geq 4$, determine the set of all  finite nonabelian simple groups $G$ such that $\mathcal{E}(G)$ is $C_n$-free. Determine the set of all  finite nonabelian simple groups $G$ such that $\mathcal{E}(G)$ is chordal.
\end{problem}

As part of our investigations of enhanced power graphs  we characterise the class of finite groups whose enhanced power graph is diamond-free, where a diamond is a graph obtained by  deleting an edge from the complete graph on four vertices. We  show in Theorem \ref{p:partition} that this class coincides with the class of finite groups whose enhanced power graph is a block graph (see Section \ref{sect:block}). Further, it turns out that this class also consists precisely of the cyclic groups together with groups which admit a nontrivial  partition consisting of cyclic subgroups.

%\begin{question}\label{Cam} \cite[Question 14]{GraphsOnGroups} For which finite groups is the enhanced power graph a perfect graph, a cograph, a chordal graph, a split graph or a threshold graph?
%\end{question}
We now refer to a graph $\Gamma$ as an \textbf{enhanced power graph} if there exists a finite group $G$ such that $\Gamma=\mathcal{E}(G)$. %In particular, as mentioned above, an enhanced power graph is a weakly perfect graph  . 
Our discussion above leads to an interesting description of various inclusions between  some main subclasses of the class of enhanced power graphs: 
\bigskip
$$C_4\hbox{-free}\supsetneq \hbox{Chordal}\supsetneq \hbox{Cographs}=\hbox{Quasi-Threshold}\supsetneq \hbox{Threshold}
$$
and 
$$
\hbox{Cographs}\supsetneq\hbox{Block}=\hbox{Diamond-free}.
$$

\bigskip\noindent
The proper inclusions are easily explained: $\mathcal{E}(S_6)$ is known to be $C_4$-free but not chordal,  (\cite[Remark 2]{Ma24}); $\mathcal{E}(A_7)$ is chordal but not a cograph (Proposition \ref{p:A7}); $\mathcal{E}(\PSL_2(q))$ is a cograph by Theorem B but it is not threshold by \cite[Theorem 1]{Ma24}; $\mathcal{E}(\Sz(q))$ is a cograph by Theorem B but it is not a block graph by Theorem \ref{p:partition} as it does not admit a partition consisting of cyclic subgroups.

We conclude this introduction with some comments about the tools used throughout the paper. A vertex of a graph is called \textbf{simplicial} if any two of its neighbours are adjacent. This concept plays a main role in the description of many classes of graphs. For instance, it is well-known that a graph is chordal if and only if every induced subgraph contains a simplicial vertex (\cite[Theorem 3.1]{farber}).
Throughout the paper, we make heavy use of the properties of simplicial vertices (Lemma \ref{l:path-simp}, Section \ref{s:simplicial}, Lemma \ref{l:cut-cyc-simp}, Lemma \ref{l:gencon1}) to derive our main results.
To the best of our knowledge, this is the first time, in the literature concerning graphs associated with groups, that  simplicial vertices have played such a central role. %are taken into consideration and their deep role recognised.
We hope that this new approach might prove effective in studying other graphs $\Gamma(G)$ associated with a group $G$. For example  we expect that the study of the set of simplicial vertices for particular choices of $\Gamma(G)$, such as the power graph or the commuting graph, could lead to some purely group theoretic results, and thus realise another part of the ambitious project outlined in \cite{BP25}. 

We inform the reader that from now on throughout the paper, the word group stands for finite group. Also, the graphs considered in this paper are always undirected  finite graphs.

%\noindent
%All groups considered in this paper are assumed to be  finite.

\section{Basic graph theoretical notions}\label{basic}

In this section we collect together notation and terminology for 
graphs that we will use in the paper.

We denote by $\mathbb{N}$ the set of positive integers and we set $\mathbb{N}_0\coloneq\mathbb{N}\cup\{0\}$. 
For $k\in \mathbb{Z}$ we set $[k]\coloneq \{x\in \mathbb{N}: x\leq k\}$ and $[k]_0 \coloneq \{x\in \mathbb{N}_0: x\leq k\}.$ 

In this paper a graph is a pair $\Gamma=(V, E)$, where $V$ is a finite non-empty set and $E\subseteq\{e\subseteq V: |e|=2\}$; 
we call $V$ the set of vertices and $E$ the set of edges of $\Gamma$.  We often write $V(\Gamma)$ for the vertex set  and $E(\Gamma)$ for the edge set of $\Gamma.$

Let $\Gamma=(V,E)$ be a graph. A graph $\Gamma'=(V',E')$ is a subgraph of $\Gamma$ if $V'\subseteq V$ and $E'\subseteq E$, and  we write $\Gamma' \subseteq \Gamma.$ The relation
$\subseteq$ is a partial order on the set of the subgraphs of $\Gamma.$
Given $S\subseteq V$, the subgraph induced  by $S$ is the graph $\Gamma[S]$ with vertex set $S$ and edge set  $\{\{x,y\}\in E:x,y\in S\}$. A subgraph $\Gamma'$ of $\Gamma$ is said to be induced if $\Gamma'=\Gamma[S]$, where $S=V(\Gamma')$, and in this case we write $\Gamma' \leq \Gamma$. For $X\subseteq V$, the \textbf{$X$-vertex-deleted subgraph} of $\Gamma$ is $\Gamma[V\setminus X]$, denoted by $\Gamma - X$; if $X=\{x\}$ we write $\Gamma -x$ instead of $\Gamma-\{x\}$.

Let $n\in \mathbb{N}.$ %The symbol $K_n$ denotes  the graph with $V(K_n)=[n]$ and  
%$E(K_n)=\{\{i,j\}:i\neq j\in [n]\}$ and is called the complete graph on $n$ vertices.
Then the graph $P_n$ with $V(P_n)=[n]$ and  
$E(P_n)=\{\{i,i+1\}:i\in [n-1]\}$ is called the \textbf{path} on $n$ vertices of length $n-1$. In particular $P_1$ is the trivial path with one vertex and no edges.
Also $C_n$ denotes the graph with $V(C_n)=[n]$ and  $E(C_n)=\{\{i,i+1\}:i\in [n-1]\}\cup\{\{n,1\}\}$ and is called the \textbf{cycle} on $n$ vertices of length $n.$

A path in $\Gamma$ is a subgraph of $\Gamma$ isomorphic to $P_n$ for some $n\in \mathbb{N}_0;$ a cycle in $\Gamma$  is a subgraph of $\Gamma$ isomorphic to $C_n$ for some $n\geq 3$. For a path $\gamma$ in $\Gamma$, an isomorphism from $P_n$ to $\gamma$ induces an ordering $x_1,\dots, x_n$ of $V(\gamma)$ which determines the structure of $\gamma$, and we denote the subgraph $\gamma$ by the ordered sequence $(x_1,\dots, x_n)$. We call $x_1$ and $x_n$ the \textbf{extreme vertices} of $\gamma=(x_1,\dots, x_n)$ and, if $n\geq 3$, the remaining vertices are called the \textbf{internal vertices} of $\gamma$.
Also, if $n\geq 2$, we call $\{x_1,x_2\}$ and $\{x_{n-1},x_n\}$ the \textbf{extreme edges} of $\gamma$, and the other edges are called \textbf{internal edges}. Note that, as subgraphs, $(x_1,\dots, x_n)=(x_n,\dots, x_1).$
Also, for $n\geq 1$, a path $\gamma\cong P_n$ in $\Gamma$ is called an \textbf{$n$-path}.

Similarly, if $\mathcal{C}$ is a cycle on $n$ vertices in a graph $\Gamma$, then an isomorphism from $C_n$ to $\mathcal{C}$ induces an ordering $x_1,\dots, x_n$ of $V(\mathcal{C})$ which determines the structure of $\mathcal{C}$.  To emphasise the presence of the edge $\{x_n, x_1\}$, we denote the subgraph $\mathcal{C}$ by the ordered sequence $(x_1,\dots, x_n, x_1)$. This representation of $\mathcal{C}$ is of course not unique, for example, as subgraphs,  $(x_1,x_2,x_3, x_4, x_1)=(x_2,x_3, x_4, x_1, x_2),$ etc. Also, for $n\geq 3$, a cycle $\mathcal{C}\cong C_n$ in $\Gamma$ is called an \textbf{$n$-cycle}.

Given a family $\mathcal{X}$ of graphs, we say that a graph is \textbf{$\mathcal{X}$-free} if none of its induced subgraphs is isomorphic to a graph in the family $\mathcal{X}$.
When $\mathcal{X}=\{\Delta\}$ we simply say that the graph is $\Delta$-free.

\begin{definition}\label{d:cograph/chordal}
{\rm A graph is called  a {\bf cograph} if it 
is $P_4$-free; {\bf chordal} if it 
is $\{C_n:n\geq 4\}$-free.
}
\end{definition}

For a graph $\Gamma=(V,E)$, let $x\in V$. The \textbf{neighborhood} of $x$ in $\Gamma$ is the set $N_\Gamma(x)=\left\lbrace y \in V | \left\lbrace y, x \right\rbrace \in E\right\rbrace$. The \textbf{degree} of $x$ in $\Gamma$, is the size of $N_\Gamma(x)$ and is denoted by $d_{\Gamma}(x).$
The \textbf{closed neighborhood} of $x$ in $\Gamma$ is   $N_\Gamma[x]=N_\Gamma(x)\cup\{x\} $. Note that computation of closed neighborhoods with respect to an induced subgraph  is straightforward: if $x\in S\subseteq V$, then $N_{\Gamma[S]}[x]=N_{\Gamma}[x]\cap S.$

 If $N_\Gamma[x]=V$, then $x$ is called a {\bf star vertex} of $\Gamma$, and the {\bf set of star vertices} of $\Gamma$ is denoted by $\mathcal{S}(\Gamma)$. If $\mathcal{S}(\Gamma)=V$, then $\Gamma$ is called a \textbf{complete graph}. The complete graph with $n$ vertices is denoted by $K_n.$
 A subset $X$ of $V$ is called a {\bf clique} if the graph induced by $X$ in $G$ is complete; a {\bf maximal clique} if $X$ is a clique and it is maximal, with respect to inclusion, among the  cliques of $G.$
\begin{definition}\label{d:simplicial}
 {\rm Let $G=(V,E)$ be a graph. A vertex $x\in V$ is called {\bf simplicial} if the induced subgraph $\Gamma[N_\Gamma[x]]$ is complete. The {\bf set of simplicial vertices} of $\Gamma$ is denoted by sl$(\Gamma)$.}   
\end{definition}

\begin{remark}\label{r:star-simplicial} Let $\Gamma=(V,E)$ be a graph. Then,
    $\mathcal{S}(\Gamma)\cap \mathrm{sl}(\Gamma)\neq\varnothing$ if and only if $\Gamma $ is a  complete graph.
\end{remark}
\begin{proof}
 Assume that $\mathcal{S}(\Gamma)\cap \mathrm{sl}(\Gamma)\neq\varnothing$ and let $z\in \mathcal{S}(\Gamma)\cap \mathrm{sl}(\Gamma)$. Pick $x,y\in V$. Then, since $z$ is a star vertex, we have $\{x,z\}, \{y,z\}\in E$. Thus, since $z$ is simplicial, we deduce $\{x,y\}\in E$. It follows that  $\Gamma$ is complete. Conversely, if $\Gamma$ is complete then $\mathcal{S}(\Gamma)=V=\mathrm{sl}(\Gamma)$ and hence $\mathcal{S}(\Gamma)\cap \mathrm{sl}(\Gamma)=V\neq\varnothing.$   
\end{proof}

\begin{lemma}\label{l:path-simp} Suppose that a graph $\Gamma$ contains either an induced path $\gamma\cong P_n$, with $n\ge3$,   or  an induced cycle  $\mathcal{C}\cong C_n$,  with $n\geq 4$.
 Then the following hold:
 \begin{itemize}
     \item[$(i)$] If $x\in V(\gamma)\cap \mathrm{sl}(\Gamma)$, then $x$ is an extreme vertex of $\gamma$.
      \item[$(ii)$] If $x\in V(\gamma)\cap \mathcal{S}(\Gamma)$, then $n=3$ and $x$ is the unique internal vertex of $\gamma$. Thus if $n\geq 4$, then we have $V(\gamma)\cap \mathcal{S}(\Gamma)=\varnothing$. 
   %\item[$(iii)$] If $\mathcal{C}$ contains a simplicial vertex of $\Gamma$, then $n=3$;
   % \item[$(iv)$] If $\mathcal{C}$ contains a star vertex of $\Gamma$, then $n=3;$
     \item[($iii)$]  $V(\mathcal{C})\cap \mathrm{sl}(\Gamma)=V(\mathcal{C})\cap \mathcal{S}(\Gamma)=\varnothing$. % (as $n\geq 4$).
     %\item[$(iv)$] $\mathcal{S}(\Gamma)\cap \mathrm{sl}(\Gamma)=\varnothing$ if and only if $\Gamma $ is not complete. }
 \end{itemize}
 \end{lemma}
\begin{proof}
%$(i)$, $(ii)$\quad Let $\gamma = (x_1,\ldots ,x_n) \cong P_n$. Suppose first that $x_i$ is a simplicial vertex of $\Gamma$. Then $\Gamma[N[x_i]]$ is a complete graph. If $1<i<n$, then $N[x_i]$ contains $x_{i-1}, x_{i+1}$ but these vertices are not adjacent in $\gamma\leq \Gamma$ and thus they are not adjacent in $\Gamma$. Hence $i\in\{1,n\}$ and $x_i$ is an extreme vertex, proving $(i)$. Suppose next that  $x_i$ is a star vertex of $\Gamma$. Then  $N[x_i]=V$. 
%If $i\geq3$ then $x_{i-2}\in V(\gamma)- N[x_i]$, and if $i\leq n-2$ then $x_{i+2}\in V(\gamma)- N[x_i]$. Thus, we necessarily have $n-1\leq i \leq 2$. As a consequence, we have $n=3$ and $i=2$ that is $x_i$ is the unique internal vertex of $\gamma$, proving $(ii)$.
Let $\gamma=(x_1,\dots, x_n)$ for suitable distinct $x_1,\dots, x_n\in V.$ 

 $(i)$  Assume to the contrary that $x_i\in V(\gamma)\cap \mathrm{sl}(\Gamma)$  for some $2\leq i\leq n-1.$ Then $x_{i-1}, x_{i+1}\in N_{\Gamma}[x_i]$ and, since $x_i$ is simplicial, we have $\{x_{i-1}, x_{i+1}\}\in E$, contradicting the fact that $\gamma$ is an induced path.
 \smallskip
 
 $(ii)$ Assume that $x_i\in V(\gamma)\cap \mathcal{S}(\Gamma)$ for some $i\in[n].$
 Then, for every $j\in [n]\setminus\{i\}$, we have $\{x_i,x_j\}\in E$, and since $\gamma$ is an induced subgraph this implies that $\{x_i,x_j\}\in E(\gamma).$ Thus $d_{\gamma}(x_i)=n-1\geq 2$. On the other hand, since the graph $\gamma$ is a path, also $d_{\gamma}(x_i)\leq 2.$ Therefore $d_{\gamma}(x_i)=n-1= 2$, so $n=3$ and $x_i$ is the unique internal vertex of $\gamma$.
\smallskip

$(iii)$ 
%Let $\mathcal{C} = (x_1, x_2,\ldots ,x_n,x_1) \cong C_n$ with $n\geq4$. Then the edges are precisely the pairs $\{ x_i, x_{i+1}\}$ writing subscripts modulo $n$. A vertex $x_i$ is not a star vertex since $\{x_i,x_{i+2}\}\not\in E$ and hence $N[x_i]\ne V$. Also $x_i$ is not simplicial since $x_{i-1}, x_{i+1}\in N[x_i]$ but $\{x_{i-1},x_{i+1}\}\not\in E$ and hence $\Gamma[N[x_i]]$ is not complete. 
   Consider $\mathcal{C}=(x_1, x_2,x_3,\dots ,x_n,x_1)\cong C_n$, with $n\geq 4$. Assume first that $\mathcal{C}$ contains a simplicial vertex $x$. As discussed above, by replacing the representation of $\mathcal{C}$ if necessary, we may assume that  $x=x_3\in$ sl$(\Gamma)$. Since $x_2, x_4\in N_{\Gamma}(x_3)$ and are distinct, we deduce that $\{x_2,x_4\}\in E$ contradicting the fact that  $\mathcal{C}$ is an induced cycle.
  Now assume that $\mathcal{C}$ contains a star vertex $x$. Then, replacing the representation of $\mathcal{C}$ if necessary, we may assume that $x=x_2\in\mathcal{S}(\Gamma).$ This implies that $\{x_2,x_4\}\in E$, again contradicting  the fact that  $\mathcal{C}$ is an induced cycle.
%\smallskip
%$(v)$ Immediate from $(iii)$ and $(iv)$.
%\smallskip
\end{proof}

%From now on let  $G$ be a finite group. 
\section{power graph and
enhanced power graph }\label{s:peg}

In this section we explore group theoretic consequences/equivalences of various graph theoretic properties of  the power graph and the enhanced power graph of a group. First we recall the definitions.

\begin{definition}\label{d:EG}{\rm 
Let $G$ be a group. 

%\smallskip
\noindent (a) The {\bf power graph} of $G$, denoted by 
$\mathcal{P}(G)$, is the graph with vertex set $G$ such that $\set{x,y}$ is an edge if and only if $x\neq y$ and for some  $m\in \mathbb{N}$, either $x=y^m$ or $y=x^m$.

%\smallskip
\noindent (b) The {\bf enhanced power graph} of $G$, 
denoted by $\mathcal{E}(G)$, is the graph with vertex 
set $G$ such that $\set{x,y}$ is an edge if and only of $x\ne y$ and $\seq{x,y}$ is cyclic. When no ambiguity arises, for $x\in G,$ we use the notation $N[x]$ instead of $N_{\mathcal{E}(G)}[x]$.} 
\end{definition}

Most of our attention will be focused on enhanced power graphs; for a recent survey on power graphs we suggest \cite{Kumar}.
In particular we note that
%\footnote{For a recent survey on power graphs we suggest \cite{Kumar}.}
$\mathcal{P}(G)$ is a spanning subgraph of $\mathcal{E}(G)$, that is, $\mathcal{P}(G)$ and $\mathcal{E}(G)$ have the same vertex set $G$ and $E(\mathcal{P}(G))\subseteq E(\mathcal{E}(G))$.
As a consequence, $\mathcal{P}(G)$ is an induced subgraph of $\mathcal{E}(G)$ if and only if $\mathcal{P}(G)=\mathcal{E}(G)$. We denote by $\mathcal{E}(G)\setminus \mathcal{P}(G)$ the graph having vertex set $G$ and edge set $E(\mathcal{E}(G))\setminus E(\mathcal{P}(G))$.
The groups for which $\mathcal{P}(G)=\mathcal{E}(G)$ were identified in \cite[Theorem 28]{alip} and are the so-called {\bf EPPO groups}, that is, finite groups in which every element has  prime power order. The EPPO groups are completely classified due to results of Higman, Suzuki and Brandl (see \cite{Ma24} for a discussion of the classification history). 

Since, for all $x,y\in G$, the subgroup $\langle x,y\rangle$ is contained in every 
subgroup of $G$ containing both $x$ and $y$, 
for each subgroup $H$ of $G$, the power graph [the enhanced power graph] of $H$ is an induced subgraph of $\mathcal{P}(G)$ [of $\mathcal{E}(G)$].  We record this basic fact  and a useful consequence.
%\D{ I have added the previous part as Cheryl asked. However,  I do not think that the results we have for general $G$ are interesting enough. I prefer to work only with $G$ finite since the beginning. I vote for removing any reference to finite/non-finite.}

\begin{lemma}\label{l:Ma24_Lem1}
Let $G$ be a  group and $H$ a subgroup of $G$. Then $\mathcal{E}(H)$   is an induced subgraph of $\mathcal{E}(G)$, and  $\mathcal{P}(H)$   is an induced subgraph of $\mathcal{P}(G)$. 
\end{lemma}

%{\rm[}$\mathcal{P}(H)${\rm ]}

\begin{corollary}\label{c:Ma24_Lem1}
Let $G$ be a group and $H\leq G$. 
If $\mathcal{E}(G)$ is a cograph {\rm[}a chordal graph, a $C_n$-free graph, for some $n\geq 4${\rm]}, then also $\mathcal{E}(H)$ is a cograph {\rm[}a chordal graph, a $C_n$-free graph, for some $n\geq 4${\rm]}. The same holds with the graphs $\mathcal{P}(G)$ and $\mathcal{P}(H)$ replacing $\mathcal{E}(G)$ and $\mathcal{E}(H)$, respectively.
\end{corollary}

We now introduce a subset, which turns out to be a subgroup (see Lemma~\ref{l:cyc}) and plays an interesting role for enhanced power graphs.
The {\bf cycliciser} of $G$ is defined by
\begin{equation}\label{e:cyc}
 \mathrm{Cyc}(G)\coloneq \{x\in G\ \vert \ \langle x,y\rangle\ \hbox{is cyclic } \hbox{for all}\  y\in G \},   
\end{equation}
and was introduced in \cite{Wepsic} under the name of cycel.
 Note that the cycliciser is not necessarily trivial. For example,
$\Cyc(Q_8)=Z(Q_8)\cong C_2$.

A cyclic subgroup of a group $G$ is called a {\bf maximal cyclic subgroup} of $G$ if it is not properly contained in another cyclic subgroup of $G$. We denote by $\mathcal{M}(G)$ the {\bf set of maximal cyclic subgroups} of $G$.   Note that $|\mathcal{M}(G)|\geq 1$ 
%\Fr{ {\bf{[We need to say that $G$ is finite for $|\mathcal{M}(G)|\geq 1$! Or at least $G$ satisfies max on the cyclic subgroups!]}}} 
and that $|\mathcal{M}(G)|= 1$ if and only if $G$ is cyclic.
These subgroups play a major role in the study of $\mathcal{E}(G)$. For example, they help us in understanding the maximal cliques of $\mathcal{E}(G)$.

\begin{lemma}\cite[Lemma 33]{alip}\label{alipur2} Let $G$ be a finite group. Then $X\subseteq G$ is a maximal clique in $\mathcal{E}(G)$ if and only if $X\in\mathcal{M}(G)$.
\end{lemma}

%\C{I think the above may be false for arbitrary groups $G$. Are there not direct limits of cyclic groups that are not cyclic and have no maximal cyclic subgroups? I feel very nervous about this unless we have a reference or proof that it is true. What does the reference prove?}
%\Fr{Yes of course! I have added finite here}\C{Thanks for this Francesco. I am very nervous about properties of infinite groups that are different from finite ones - I'm not very confident tere.}

An element $x$ of $G$ is called a {\bf maximal element} if $\langle x\rangle\in \mathcal{M}(G).$ We denote the {\bf set of maximal elements} of $G$ by
 \begin{equation}\label{e:M}
     M(G) := \{ x\in G \mid \langle x\rangle\in \mathcal{M}(G)\}. 
 \end{equation}
We summarise some basic properties of the cycliciser  and the enhanced power graph in the following lemma. 

\begin{lemma}\label{l:cyc}    
Let $G$ be a group, and $\mathrm{Cyc}(G)$, $M(G)$ as in \eqref{e:cyc}, \eqref{e:M}. Then 
\begin{enumerate}
    \item[$(i)$] $ \mathrm{Cyc}(G) = \mathcal{S}(\mathcal{E}(G))$;
    \item[$(ii)$] $ \mathrm{Cyc}(G) = \cap_{H\in \mathcal{M}(G)} H$  is a cyclic subgroup of the centre of $G$; %\D{What do you think to write instead: $\mathrm{Cyc}(G) = \cap_{H\in \mathcal{M}(G)} H$? }\C{Happy to have this - I changed it. I also corrected the mis-spelling of `centre'.}
    \item[$(iii)$] $\mathcal{E}(G)$ is a complete graph if and only if $G$ is cyclic.
\end{enumerate}
\end{lemma}

\begin{proof}
     $(i)$ It follows from the definitions of adjacency in $\mathcal{E}(G)$ and of a star vertex.
     
     $(ii)$ %It follows from \eqref{e:cyc} that $\mathrm{Cyc}(G)$ is contained in each maximal cyclic subgroup of $G$, so $ \mathrm{Cyc}(G) \leq \cap_{u\in M(G)} \langle u\rangle$. Conversely let $x\in \cap_{u\in M(G)} \langle u\rangle$ and let $y\in G$. Then $y$ is contained in some maximal cyclic subgroup $\langle z\rangle$ of $G$, so $z\in M(G)$ and $x\in\langle z\rangle$. Thus $x, y\in \langle z\rangle$ and hence $\langle x,y\rangle$ is cyclic, so $x\in \mathrm{Cyc}(G)$. Therefore $ \mathrm{Cyc}(G) = \cap_{u\in M(G)} \langle u\rangle$. In particular, $\mathrm{Cyc}(G)$ is a subgroup and is cyclic, and we observed above that  $\mathrm{Cyc}(G)\subseteq Z(G)$.%, and in particular, $\mathrm{Cyc}(G)$ is normal in $G$. 
     This is  a synthesis of the content of \cite[Section 2.1]{Wepsic}.

     $(iii)$ If $G$ is cyclic then $\mathrm{Cyc}(G)=G$  by \eqref{e:cyc}, and hence $\mathcal{E}(G)$ is complete. Conversely, assume that $\mathcal{E}(G)$ is complete. Then $\mathrm{Cyc}(G)=G$  by \eqref{e:cyc}, and hence by $(ii)$, $G$ is cyclic.
\end{proof}

Next we give a combinatorial criterion, in terms of properties of $\mathcal{E}(G)$, for membership of $M(G)$.

\begin{lemma}\label{l:easy} 
Let $G$ be a group, let $x\in G$, and let $M(G)$ be as in \eqref{e:M}. Then 
$x \in M(G)$ if and only if the closed neighbourhood  $N[x]$ is 
$\langle x\rangle.$
\end{lemma}

 \begin{proof}  
It follows from the definition of $\mathcal{E}(G)$ that $\langle x\rangle \subseteq N[x]$. Suppose first that $x \in M(G)$, and let $y\in N[x]$. Then $\langle y,x\rangle$ is cyclic and contains $ \langle x\rangle$. As $x \in M(G)$, $\langle x\rangle$ is a maximal cyclic subgroup, and hence $\langle y,x\rangle= \langle x\rangle$, so $y\in \langle x\rangle$. Thus $N[x]= \langle x\rangle.$
Conversely, suppose that $N[x]= \langle x\rangle$, and let $\langle y\rangle$ be a maximal cyclic subgroup containing $x$.  Then $y\in N[x]= \langle x\rangle.$ Thus $\langle y\rangle\subseteq \langle x\rangle$, and from the maximality of $\langle y\rangle$ we deduce that $\langle y\rangle= \langle x\rangle$. Thus $x\in M(G)$.
\end{proof}

\subsection{Simplicial vertices in enhanced power graphs.} \label{s:simplicial} 

We start this section by obtaining both graph theoretic  and group 
theoretic  necessary and sufficient conditions for a vertex of $\mathcal{E}(G)$ to be simplicial (see Definition \ref{d:simplicial}).

\begin{lemma}\label{l:simplicial}
Let $G$ be a group and $x\in G$. Then the following are equivalent.
\begin{enumerate}
\item[$(a)$] 
$x\in \mathrm{sl}(\mathcal{E}(G))$;
\item[$(b)$] $x$ lies in a unique maximal cyclic subgroup 
of $G$;
\item[$(c)$] $N[x]$ is a maximal 
cyclic subgroup of $G$;
\item[$(d)$] $N[x]=\langle y \rangle = N[y]$, for some $y\in M(G)$;
\item[$(e)$] $N[x]$ is a  
cyclic subgroup of $G$.
\end{enumerate}
\end{lemma}

\begin{proof}
$(a)\Rightarrow (b)$.\quad Let $x\in \mathrm{sl}(\mathcal{E}(G))$. Then $N[x]$ is a clique in $\mathcal{E}(G)$. Suppose that $\langle y_1\rangle$ and $\langle y_2\rangle$ are maximal cyclic subgroups of $G$, both containing $x$. Then $y_1,y_2\in N[x]$ and, since $N[x]$ is a clique, $y_1, y_2$ are adjacent in $\mathcal{E}(G)$. Thus $\langle y_1,y_2\rangle$ is cyclic. The maximality of $\langle y_1\rangle$ and $\langle y_2\rangle$ then implies that 
$\langle y_1\rangle=\langle y_2\rangle$.

\smallskip\noindent
$(b)\Rightarrow (c)$.\quad Assume that $\langle y\rangle$ is the unique maximal cyclic subgroup of $G$ containing $x$.  Then $\langle x,y'\rangle$ is cyclic for each $y'\in \langle y\rangle$, and hence $\langle y\rangle\subseteq N[x]$. On the other hand, if $z\in N[x]$, then $\langle z,x\rangle$ is cyclic, and the uniqueness of  $\langle y\rangle$ implies that $\langle z,x\rangle\leq \langle y\rangle$. 
Hence $N[x]\subseteq \langle y\rangle$, and so equality holds, proving $(c)$.

\smallskip\noindent
$(c)\Rightarrow (d)$.\quad Assume that $N[x]= \langle y\rangle$ is a maximal cyclic subgroup of $G$. Then  $y\in M(G)$ by \eqref{e:M}, and so, by Lemma \ref{l:easy}, $N[y]=\langle y \rangle=N[x]$.

\smallskip\noindent
$(d)\Rightarrow (e)$.\quad  This holds trivially.

\smallskip\noindent
$(e) \Rightarrow (a)$.\quad Assume that $N[x]$ is a cyclic subgroup of $G$. Then any two elements of $N[x]$ generate a cyclic subgroup, so $N[x]$ is a clique of $\mathcal{E}(G)$ and hence $x$ is a simplicial element of $\mathcal{E}(G)$.
\end{proof}

By Lemma~\ref{l:simplicial}, the identity element $1$ is simplicial if and only if $G$ is cyclic. Also if $G$ is cyclic then every vertex is simplicial. Further, $\textrm{sl}(\mathcal{E}(G))=G\setminus\{1\}$  if and only if  $G$ is not cyclic and every element in $G\setminus\{1\}$ is contained in a unique maximal cyclic subgroup of $G$. This holds, for example,  if $G\cong C_p^k,$ with $p$ prime and $k\geq 2$, and if $G=A_5$. This interesting situation will be analyzed in detail in Theorem \ref{p:partition}. We note here one further consequence.

\begin{corollary}\label{c:estremiG} 
Let $G$ be a  group. Then $M(G)\subseteq \mathrm{sl}(\mathcal{E}(G))$.
\end{corollary}
%
% old proof
%
\begin{proof} Let $x\in M(G)$. Then, by Lemma~\ref{l:easy},  $N[x]= \langle x\rangle,$ so condition $(c)$ in Lemma~\ref{l:simplicial} holds, and hence also condition $(a)$ in Lemma~\ref{l:simplicial} holds, that is, $x\in \mathrm{sl}(\mathcal{E}(G)).$ 
\end{proof}

In particular, $\mathrm{sl}(\mathcal{E}(G))\neq \varnothing$ for all groups $G$.  However, in general the set $M(G)$ is not 
equal to sl$(\mathcal{E}(G))$. For 
example, if $G$ is cyclic and nontrivial,  
then $G=\mathrm{sl}(\mathcal{E}(G))$ by Lemma~\ref{l:simplicial}, while $M(G)$ consists of the generators of the maximal cyclic subgroups, and is a proper subset of $G$. Also, if $G$ is non-cyclic, then the subset $M(G)$ can be rather smaller than 
$\mathrm{sl}(\mathcal{E}(G))$. For instance, in
the abelian group $G=C_2\times C_2\times C_3$ the 
set $M(G)$ has cardinality $6$ since the 
maximal cyclic subgroups of $G$ are the three 
subgroups of order $6$, while the set $\mathrm{sl}(\mathcal{E}(G))$  has cardinality $9$
because, by Lemma \ref{l:simplicial}, it contains also the three involutions of $G$.
Indeed, the unique maximal cyclic 
subgroup of $G$ containing an involution $y$ is the one generated by $y$ 
and by the unique subgroup of $G$ of order $3$. 

Observe that a very manageable non-empty subset of maximal elements in a group $G$ is the one comprising elements with order maximal  with respect to  divisibility. 

We also observe that if $x\in G$ has cyclic centraliser $C_G(x)$, then $x$ is simplicial. Indeed in this case $C_G(x)$ is the only maximal cyclic subgroup of $G$ containing $x$, and hence $x\in \mathrm{sl}(\mathcal{E}(G))$ by Lemma \ref{l:simplicial}.

\subsection{The role of the cycliciser for the enhanced power graph.} We explore the link between 
$\mathcal{E}(G)$ and $\mathcal{E}(G/C)$, for a  group $G$, where $C=\Cyc(G)$ with $\Cyc(G)$  as in \eqref{e:cyc}. 
We often use the `bar notation', that is, we denote the quotient group $G/C$ by $\overline{G}$, and a similar notation is used to denote elements and  subgroups of $\overline{G}$. 
We also denote by $\pi$ the natural projection $\pi:G\to \overline{G}.$ 

\begin{lemma}\label{l:E(G/Cyc):1}
Let $G$ be a group and let $g,h\in G$. Then the subgroup 
$\seq{g,h}$ is cyclic if and only if  
$\seq{gc_1,hc_2}$ is cyclic for all $c_1,c_2\in \Cyc(G)$.
\end{lemma}

\begin{proof}
Let $C:= \Cyc(G)$. Assume first that $\seq{g,h}$ is cyclic, and let $c_1,c_2\in C$. 
Let $x\in M(G)$ be such that $\seq{g,h}\leq \seq{x}$. Then $C\leq \seq{x}$, by Lemma \ref{l:cyc}$(ii)$, and therefore 
$$\seq{gc_1,hc_2}\leq \seq{g,h,c_1,c_2}\leq \seq{g,h,x}\leq \seq{x},$$
so $\seq{gc_1,hc_2}$ is cyclic. The reverse implication is obvious.
\end{proof}

\begin{lemma}\label{l:E(G/Cyc):2}
Let $G$ be a group and let $C=\Cyc(G)$ and $\overline{G}=G/C$. Then 
%\begin{enumerate}
%\item[$(a)$]     If $D\leq G$ is such that $\overline{D}$ is a cyclic subgroup of $\overline{G}$, then $D$ is cyclic.
%    \item[$(b)$] 
for all $x,y\in G$ such that $\overline{x}\neq \overline{y}$, 
\[
\{x,y\}\in E(\mathcal{E}(G)) \quad \mbox{if and only if } \quad 
\{\overline{x},\overline{y}\}\in E(\mathcal{E}(\overline{G})).
\]
\end{lemma}

\begin{proof}
We first show that if $D\leq G$ is such that $\overline{D}$ is a cyclic subgroup of $\overline{G}$, then $D$ is cyclic.
Assume that $\overline{D}=\seq{\overline{g}}$ for some $\overline{g}\in \overline{G}$. Then $D\leq \seq{g}C$. Let $h\in M(G)$ be such that $\seq{g}\leq \seq{h}$. Then $C\leq \seq{h}$, by Lemma \ref{l:cyc}$(ii)$, and hence $D\leq  \seq{h}$, so $D$ is  cyclic.

Now let $x,y\in G$ such that 
$\overline{x}\neq \overline{y}$, and let  
$D=\langle x,y\rangle$ and $\overline{D} =
\langle \overline{x},\overline{y}\rangle$. 
If $D$ is cyclic then clearly $\overline{D}$ 
is cyclic, while the converse holds by the  
previous paragraph.
% Part (b) now follows from Definition~\ref{d:EG}. 
\end{proof}
Considering, for a group $G$, the graph $\widehat{\mathcal{E}}(G)$ obtained from $\mathcal{E}(G)$ by adding a loop at each vertex,   Lemma 3.10 shows that the natural projection map $G \to C/\Cyc(G)$ induces a surjective  graph homomorphism $\widehat{\mathcal{E}}(G) \to \widehat{\mathcal{E}}(G/\Cyc(G))$.

\noindent
Recall now that, given a graph $\Gamma=(V,E)$, two
distinct vertices $x,y\in V$ such that $N_\Gamma[x]=N_\Gamma[y]$ are called {\bf closed twins}.
In particular, if $x, y\in V$ are closed twins, then $\{x,y\}\in E$.
The following result is in the spirit of \cite[Lemma 5]{Ma24}.

\begin{prop}\label{p:E(G/Cyc)}
Let $\Delta $ be a graph without closed twins and let $G$ be a group. If $\mathcal{E}(G/\Cyc(G))$ 
is $\Delta$-free then also $\mathcal{E}(G)$ is $\Delta$-free.
\end{prop}

\begin{proof}
Let $C=\Cyc(G)$, $\overline{G}=G/C$, $\Gamma=\mathcal{E}(G)$ and $\overline{\Gamma}=\mathcal{E}(\overline{G})$, and let $\pi:G\to \overline{G}$ be the natural projection.
Assume that $\overline{\Gamma}$ is $\Delta$-free, and assume  to the contrary that there exists $X\subseteq G$ such that the induced subgraph $\Gamma[X]\cong \Delta$. 
 
\noindent\emph{Claim: $X\cap xC = \{x\}$ for all $x\in X$.}\quad 
Suppose to the contrary that  $x\in X$ and $X\cap xC \neq \{x\}.$ Then there exists $c\in C\setminus\{1\}$ such that $xc\in X$. Now, $\langle x,xc\rangle=\langle x,c\rangle$ and this subgroup is cyclic by \eqref{e:cyc}. Also $x\neq xc$, and hence $\{x,xc\}$ is an edge of $\Gamma$. Further, for every $y\in G\setminus \{ xc, x\}$, we have $y\in N_\Gamma[x]$ if and only if $\langle x,y\rangle$ is cyclic, and by Lemma~\ref{l:E(G/Cyc):1} this holds if and only if  $\langle xc,y\rangle$ is cyclic, which is equivalent to $y\in N_\Gamma[xc]$. Thus 
\[
 \{x\} \cup N_\Gamma(x) = N_\Gamma[x] =  N_\Gamma[xc] = \{xc\}\cup N_\Gamma(xc).
\]
This implies, since $x, xc\in X$, that $N_{\Gamma[X]}[x] = N_{\Gamma[X]}[xc],$
and hence $x, xc$ are closed twins in 
$\Gamma[X]\cong \Delta$, which is a contradiction, proving the Claim.
\smallskip

It follows immediately from the Claim that the restriction  $\pi_{|X}$ is one-to-one. Hence, by Lemma~\ref{l:E(G/Cyc):2}, $\pi$ defines a graph isomorphism from $\Gamma[X]$ to $\overline{\Gamma}[\pi(X)].$ Thus $\overline{\Gamma}[\pi(X)]\cong \Gamma[X]\cong \Delta$, which contradicts the assumption that $\overline{\Gamma}$ is $\Delta$-free, and completes the proof.
\end{proof}
%The above proposition is particularly interesting for us, once observed that $C_n$, for $n\geq 4,$ as well as $P_4$ do not admit closed twins. 

\begin{corollary}\label{c:E(G/Cyc)} Let $G$ be a group and let $n\geq 4$.
If $\mathcal{E}(G/\Cyc(G))$ is a cograph {\rm[}a chordal graph, a $C_n$-free graph {\rm]}, then also $\mathcal{E}(G)$ is a cograph {\rm[}a chordal graph, a $C_n$-free graph{\rm]}.     
\end{corollary}

\begin{proof}
This is a consequence of Proposition \ref{p:E(G/Cyc)} and the fact that the graphs $C_n$ and $P_n$ have no closed twins for any $n\geq 4$.    
\end{proof}
\noindent

\section{Chordal and cograph properties of enhanced power graphs}\label{s:ccprops}

We begin this section by proving Theorem~A, establishing that if the enhanced power graph of a  group is a cograph, then it is chordal also.

% \begin{theorem_A} Let $G$ be a group such that $\mathcal{E}(G)$  is a cograph. Then $\mathcal{E}(G)$  is chordal.
% \end{theorem_A}
\medskip\noindent
\emph{Proof of Theorem A.}\quad  Let  $\mathcal{E}(G)=(G,E)$, for a group $G$, and suppose that $\mathcal{E}(G)$  is a cograph.  Assume, for a contradiction, that $\mathcal{E}(G)$  is not chordal, and let $\mathcal{C}$ be an induced $n$-cycle of $\mathcal{E}(G)$, for some $n\geq4$.   If $n\geq 5,$ then any set of four consecutive vertices in $\mathcal{C}$ induces a subgraph of $\mathcal{E}(G)$ isomorphic to $P_4$, which is a contradiction.  Thus $n=4$, and  $\mathcal{C}=(x_1,x_2,x_3,x_4,x_1),$ for pairwise distinct vertices $x_1,x_2,x_3,x_4\in G.$
By Lemma \ref{l:path-simp}\, $(iii),$ no vertex of $\mathcal{C}$ is a simplicial vertex of $\mathcal{E}(G)$, and hence, by 
Corollary~\ref{c:estremiG},  no vertex of $\mathcal{C}$ is maximal. 
Since $\set{x_1,x_4}\in E$, the subgroup $\langle x_1, x_4\rangle$ is cyclic. Let $\langle z\rangle$ be a maximal cyclic subgroup of $G$ containing $\langle x_1, x_4\rangle$. Then $z$ is maximal, and hence  $z\not\in\{x_1,x_2,x_3,x_4\}$. In particular  $|\{z, x_1, x_2, x_3\}|=4$.
Now  $\set{z,x_1}\in E$ since $\langle z, 
x_1\rangle = \langle z\rangle$ is cyclic  and so  
$\gamma=(z,x_1,x_2,x_3)$ is a path in $\mathcal{E}(G).$
We claim that  $\gamma$ is in fact an induced path. First note that $\set{x_1,x_3}\notin E$ since $\mathcal{C}$ is an induced subgraph. Assume now, to the contrary, that $\set{z,x_2}\in E$. Then $\langle z, x_2\rangle$ is cyclic and the maximality of $\langle z\rangle$ implies that $x_2\in \langle z\rangle.$ Since also $ x_4\in \langle z\rangle,$ we deduce that $\langle x_2, x_4\rangle$ is cyclic, and hence $\set{x_2,x_4} \in E$, contradicting the fact that $\mathcal{C}$ is an induced subgraph.
Assume next that $\set{z,x_3}\in E$. Then, $\langle z, x_3\rangle$ is cyclic and the maximality of $\langle z\rangle$ implies $x_3 \in \langle z\rangle.$ Since also $x_1\in \langle z\rangle$ we deduce that $\set{x_3,x_1} \in E,$ contradicting the fact that $\mathcal{C}$ is induced. 
%It remains to show that,  for $i\in \{ 2,3\}$, $\set{z,x_i}\notin E$.  Assume to the contrary that $\set{z,x_i}\in E$. Then $\langle z, x_i\rangle$ is cyclic and the maximality of $\langle z\rangle$ implies that $x_i\in \langle z\rangle.$ Since also $x_1, x_4\in \langle z\rangle, $ we deduce that $\langle x_i, x_1\rangle$ and $\langle x_i, x_4\rangle$ are both cyclic, and hence both $\set{x_i,x_1}, \set{x_i,x_4} \in E$, contradicting the fact that $\mathcal{C}$ is an induced subgraph. 
Thus we have proved the claim that $\gamma$ is an induced $4$-path $P_4$, which contradicts the fact that $\mathcal{E}(G)$ is a cograph. This completes the proof. \quad \qed

\bigskip

A natural question one might ask is whether the implication in Theorem A holds for other graphs associated with  groups, in particular, whether it holds for power graphs. 
This holds trivially for EPPO groups since, as noted in Section~\ref{s:peg}, these are the groups for which the power graph and the enhanced power graph coincide. Also, the list of simple groups for which the power graph is a cograph was determined in \cite[Theorem 1.3 and Theorem 5.1]{CA}, and is a sublist  of the list of simple groups for which the power graph is  chordal  \cite[Theorem 1.1]{BRA}. These special cases suggest the possibility that the implication may hold for all finite groups.

\begin{problem} Determine all finite groups $G$ such that $\mathcal{P}(G)$ is both a cograph and a chordal graph. Is it true that, if $\mathcal{P}(G)$ is a cograph, then $\mathcal{P}(G)$ is  chordal?
\end{problem}

\subsection{Quasi-threshold graphs and the group theoretical characterisation of groups whose enhanced power graph is a cograph.}\label{s:threshold}

As mentioned in the introduction, Theorem A was for us  
an unexpected result, since  from the graph 
theoretic point of view the 
class of graphs which are both cographs 
and chordal is considered to be quite 
restricted, namely it is the class of {\bf quasi-threshold 
graphs} (\cite[Theorem 3]{Ya}), so we have the following consequence of Theorem A. 

\begin{corollary}\label{thresh} Let $G$ be a group. The following facts are equivalent:
\begin{itemize}
    \item [$(a)$] $\mathcal{E}(G)$ is a cograph;
    \item[$(b)$] $\mathcal{E}(G)$ is a quasi-threshold graph.
\end{itemize}
\end{corollary}

Among the quasi-threshold graphs are the {\bf threshold graphs} defined as  the $\{P_4, C_4, 2K_2\}$-free graphs (see \cite{Maha}), where $2K_2$ denotes  the graph with $V(2K_2)=[4]$ and  
$E(2K_2)=\{\{1,2\}, \{3,4\} \}$. As mentioned Section~\ref{s:intro},   the  groups $G$ 
for which $\mathcal{E}(G)$ is a threshold graph are the 
cyclic groups, the dihedral groups and the 
elementary abelian $2$-groups (\cite[Theorem 1]{Ma24}), while  the class of  groups  $G$ 
for which $\mathcal{E}(G)$ is a quasi-threshold graph contains two infinite families of finite nonabelian simple groups (Theorem B). We obtain a characterisation in terms of properties of maximal cyclic subgroups of a group $G$ for $\mathcal{E}(G)$ to be a quasi-threshold graph (Proposition~\ref{p:W}). For this we use the following notation, for a graph $\Gamma$:
\begin{itemize}
    \item $m(\Gamma)$  denotes the number of maximal cliques of  $\Gamma$, and 
    \item $\alpha(\Gamma)$ denotes the maximum size of a stable set of $\Gamma$, that is, a set of pairwise non-adjacent vertices.
\end{itemize}

\noindent
Since each maximal clique contains at most one vertex of a stable set, the inequality $\alpha(\Gamma)\leq m(\Gamma)$ holds for every graph $\Gamma$. It was shown in \cite[Theorem 3]{Ya} that $\Gamma$ is a quasi-threshold graph if and only if $\alpha(\Gamma')= m(\Gamma')$ for every subgraph $\Gamma'\leq \Gamma$.
Another more local characterisation was given in \cite[Theorem 3]{Ya}: $\Gamma$ is a quasi-threshold graph if and only if, for each edge $\{x,y\}\in E(\Gamma)$, either $N_{\Gamma}[x]\subseteq N_{\Gamma}[y]$ or $N_{\Gamma}[y]\subseteq N_{\Gamma}[x]$.

We explore the consequences of these facts for the enhanced power graph of a  group $G$. Observe first that $m(\mathcal{E}(G))=|\mathcal{M}(G)|$, by Lemma \ref{alipur2}. Moreover, a vertex-subset consisting of one generator of each of the maximal cyclic subgroups is clearly a stable set of $\mathcal{E}(G)$, and thus $\alpha(\mathcal{E}(G))\geq m(\mathcal{E}(G))$. Thus, by the inequality noted in the previous paragraph, we have $\alpha(\mathcal{E}(G))=m(\mathcal{E}(G))$. By the same argument, for every subgroup $H\leq G$, we have $\alpha(\mathcal{E}(H))=m(\mathcal{E}(H))$.  Thus,  the condition  characterising $\mathcal{E}(G)$ quasi-threshold, in the first characterisation in the previous paragraph, is  satisfied for the subgraphs induced by subgroups of $G$. This discussion therefore yields the following proposition.

\begin{prop}\label{enh-thresh} Let $G$ be a  group. Then the following are equivalent:
\begin{itemize}
    \item [$(a)$] $\mathcal{E}(G)$ is a cograph;
    \item[$(b)$] $\alpha(\mathcal{E}(G)[X])=m(\mathcal{E}(G)[X])$ for all $X\subseteq G$;
    \item[$(c)$] $\alpha(\mathcal{E}(G)[X])=m(\mathcal{E}(G)[X])$ for all $X\subseteq G$ such that $X$ is not a subgroup;
    \item[$(d)$] for all $x,y\in G$, if $\seq{x,y}$ is cyclic then either 
    $$N[x]=\{z\in G: \seq{z,x} \hbox{ is cyclic}\}\subseteq \{z\in G: \seq{z,y} \hbox{is cyclic}\}=N[y]$$
or 
    $$N[y]=\{z\in G: \seq{z,y} \hbox{ is cyclic}\}\subseteq \{z\in G: \seq{z,x} \hbox{is cyclic}\}=N[x].$$
\end{itemize}
\end{prop}

We use Proposition 
\ref{enh-thresh} to prove Proposition~\ref{p:W}, which provides a fundamental and 
manageable group theoretic criterion for the enhanced power graph to be a cograph.

\begin{prop}\label{p:W}
For a group $G$, the following are equivalent.
\begin{enumerate}
    \item[$(a)$]  $\mathcal{E}(G)$ is a cograph;
    \item[$(b)$] %either $G$ is cyclic, or $G$ is non-cyclic and,
    for all pairwise distinct maximal cyclic subgroups $A, B, C$ of $G$,  either $A\cap B\leq A\cap C$ or 
$A\cap C\leq A\cap B$.    
\end{enumerate}
\end{prop}

\begin{proof}
$(a)\Rightarrow (b).$ Assume that  $\mathcal{E}(G)$ is a cograph and %that $G$ is non-cyclic. Assume also,
assume to the contrary that $G$ has three pairwise distinct 
maximal cyclic subgroups $A=\seq{a}, B=\seq{b}$ and $C=\seq{c}$ such  that
$A\cap B\not\leq A\cap C$ and  
$A\cap C\not\leq A\cap B$. 
Let $A\cap B=\seq{x}$ and $A\cap C=\seq{y}$. Note that $A\cap C\not\leq A\cap B$ implies $y\notin B$.  
Now $b\in B\subseteq N[x]$. Suppose that $b\in
N[y]$. Then $\seq{b,y}$ is cyclic and so, by the maximality of $B$, we have $\seq{b,y}=\seq{b}=B$, and hence $y\in B$, which is a contradiction. Therefore $b\in N[x]\setminus  N[y]$. Similarly, using $A\cap B\not\leq A\cap C$,  we deduce that $c\in N[y]\setminus N[x]$. Since $x, y\in A$ the subgroup $\seq{x,y}$ is cyclic, and we have a contradiction by Proposition \ref{enh-thresh}.
\smallskip 

$(b)\Rightarrow (a).$ Suppose that $G$ satisfies condition $(b)$ % If $G$ is cyclic then $\mathcal{E}(G)$ is a complete graph, and hence a cograph. Let next $G$ be non-cyclic. %Then by Remark~\ref{r:cog}, $G$ has at least three pairwise distinct maximal cyclic subgroups, and the condition in $(b)$ holds for all pairwise distinct maximal cyclic subgroups $A, B, C$. 
and assume to the contrary that  $\mathcal{E}(G)$ is not a cograph. Then, in particular, $\mathcal{E}(G)$ is not a complete graph, and so, by Lemma~\ref{l:cyc}, $G$ is not cyclic.
By Proposition \ref{enh-thresh}, there exist 
$x,y\in G$ such that $\seq{x,y}$ is cyclic, $N[x]\not\subseteq N[y]$ and $N[y]\not\subseteq N[x]$. Let $A$ be a maximal cyclic subgroup of $G$ containing $\seq{x,y}$. 
Since  $N[x]\not\subseteq N[y]$ we  may choose an element $b\in N[x]\setminus N[y]$. Then  $\seq{b,x}$ is cyclic and $\seq{b,y}$ is not. Let $B$ a maximal cyclic subgroup of $G$ containing $\seq{b,x}$. Then $B$ does not contain $y$, otherwise $\seq{b,y}\leq B$ would be cyclic. as a consequence, we have  $B\neq A$ and $B\cap A$ contains $x$ but  does not contain $y$. Similarly, since  
$N[y]\not\subseteq N[x]$, we may choose an element $c\in N[y]\setminus N[x]$ and $C$ a maximal cyclic subgroup of $G$ containing $\seq{c,y}$. Since $\seq{c,x}$ is not cyclic, $C$ does not contain $x$. In particular, $C\neq A$, $C\neq B$. Moreover, $C\cap A$ contains $y$ and does not contain $x$. Therefore $A\cap B\not\leq A\cap C$ and $A\cap C\not\leq A\cap B$, which is a contradiction.
\end{proof}

The condition in Proposition \ref{p:W}$(b)$ is vacuously true if the group $G$ is cyclic, but  needs to be checked if $G$ is non-cyclic. In this regard we note the following remark.

\begin{remark}\label{r:cog}
A finite non-cyclic group has at least three pairwise distinct maximal cyclic subgroups. Non-cyclic groups with exactly three maximal cyclic subgroups exist, for example, $C_2\times C_2$.
\begin{proof}
    Let $G$ be a non-cyclic finite group and let $M_1=\seq{a}\in\mathcal{M}(G).$ Then $M_1\ne G$ and we choose $g\in G\setminus M_1$. Let $M_2\in \mathcal{M}(G)$ containing $g$. Then $M_2\ne M_1$ as $g\in M_2\setminus M_1$. Also $ga\not\in M_1$ as otherwise $g\in M_1$; and $ga\not\in M_2$ as otherwise $a\in M_2$ and $M_2$ is a cyclic subgroup properly containing $M_1=\seq{a}$, which is a contradiction. Thus $ga\notin M_1\cup M_2$ and so a maximal cyclic subgroup $M_3$ containing $ga$ is distinct from both $M_1$ and $M_2$.
\end{proof}
\end{remark}

We now give a more flexible criterion, in line with 
\cite[Proposition 1]{Ma24}.
It will prove useful 
in dealing with the finite nonabelian simple groups.

\begin{corollary}\label{c:simil_Ma} 
Let $G$ be a non-cyclic  group and let $p$ be a prime. 
Suppose that, for all distinct $H, K\in \mathcal{M}(G),$ there exists $a\geq0$ such that 
$|H\cap K|=p^a$. 
Then the cycliciser $\mathrm{Cyc}(G)$ is a cyclic $p$-group 
and $\mathcal{E}(G)$ is a cograph.
\end{corollary}

\begin{proof}
Since $G$ is not cyclic, there exist distinct 
subgroups $A, B\in \mathcal{M}(G)$, and by assumption, 
$A\cap B$ is a cyclic $p$-group. Hence, by 
Lemma~\ref{l:cyc}$(ii)$, $\Cyc(G)$ is a cyclic $p$-group. 

Now assume to the contrary that $\mathcal{E}(G)$ is not a cograph. Then by Proposition \ref{p:W} there exist three distinct maximal cyclic subgroups $A,B,C$ of $G$ such that neither $A\cap B\leq A\cap C$, nor $A\cap C\leq A\cap B$.
On the other hand, by assumption, $A\cap C$ 
and $A\cap B$ are $p$-subgroups of the unique cyclic Sylow $p$-group of $A$, and hence the smaller of $A\cap C$ 
and $A\cap B$ is contained in the 
other. This is a contradiction, and hence  
$\mathcal{E}(G)$ is a cograph.
\end{proof}

\subsection{Enhanced power graphs of nilpotent groups}
We now apply Theorem A 
 to  nilpotent groups obtaining, in  Theorem~\ref{t:Ma24_Thm2},  a 
simplified proof of a result of Ma et al. \cite[Theorem 2]{Ma24}, and also the following new observation which follows from Theorem~\ref{t:Ma24_Thm2}:  

\begin{center}
\emph{For nilpotent groups $G$, $\mathcal{E}(G)$ is chordal if and only if $\mathcal{E}(G)$ is $C_4$-free.}    
\end{center}

\begin{lemma}\label{l:p-groups}
If $G$ is a $p$-group, for some prime 
number $p$, then $\mathcal{E}(G)$ is a cograph.
\end{lemma}
\begin{proof}
This is clear if $G$ is cyclic and, when $G$ is not cyclic, the proof is a direct consequence of Corollary \ref{c:simil_Ma}.
%
%   old proof
%
%Suppose to the contrary that  $\gamma=(x_1,x_2,x_3,x_4)$ is an induced $4$-path  of $\mathcal{E}(G)$.  In particular, $\{ x_2,x_3\}$ is an edge of $\mathcal{E}(G)$ so $\langle x_2,x_3\rangle$ is a cyclic $p$-group and its subgroups form a chain. Therefore either $\langle x_2\rangle\leq \langle x_3\rangle$ or $\langle x_3\rangle\leq \langle x_2\rangle$. Since we may also write $\gamma=(x_4, x_3, x_2, x_1)$,  we may assume without loss of generality that $\langle x_2\rangle\leq \langle x_3\rangle$. This implies  that $\langle x_2,x_4\rangle\leq \langle x_3,x_4\rangle$, and since $\langle x_3,x_4\rangle$ is cyclic, it follows that  $\langle x_2,x_4\rangle$ is cyclic. Thus $\{ x_2,x_4\}$ is an edge of $\mathcal{E}(G)$, contradicting the fact that $\gamma$ is an induced path.
\end{proof}

 \begin{theorem}{\rm \cite[Theorem 2]{Ma24}}\label{t:Ma24_Thm2}
    Let $G$ be a finite nilpotent group. Then the following are equivalent.
    \begin{itemize}
    \item[$(a)$]$\mathcal{E}(G)$ is a cograph;
    \item[$(b)$]$\mathcal{E}(G)$  is a chordal graph;
    \item[$(c)$]$\mathcal{E}(G)$  is
    $C_4$-free;
    \item[$(d)$] $G$ has at most one non-cyclic Sylow subgroup;
    \item[$(e)$] $G= P\times C_n$, with $P$  a $p$-group and $\gcd(n,p)=1$.
    \end{itemize}
\end{theorem}

\begin{proof}
$(a)\Rightarrow (b)$.\quad This follows from Theorem A.

\noindent
$(b)\Rightarrow (c)$.\quad This follows from the definition of a chordal graph.

\noindent
$(c)\Rightarrow (d)$.\quad 
Assume that $\mathcal{E}(G)$  is $C_4$-free. Suppose to the contrary that $p,q$ are distinct primes dividing $|G|$ such that the Sylow $p$-subgroup, and the Sylow $q$-subgroup of $G$ are both non-cyclic. Then there are four elements $x_1,x_2,y_1,y_2$ with $x_1,x_2$ both $p$-elements such that $\langle x_1,x_2\rangle$ is not cyclic, and $y_1,y_2$ both $q$-elements such that $\langle y_1,y_2\rangle$ is not cyclic. It follows that  $(x_1,y_1,x_2,y_2,x_1)$ is an induced $4$-cycle in $\mathcal{E}(G)$, which is a contradiction.

\noindent
$(d)\Rightarrow (e)$.\quad This follows immediately  from the fact that $G$ is the direct product of its Sylow subgroups.

\noindent
$(e)\Rightarrow (a)$.\quad For $G$ as in $(e)$, clearly the subgroup $C_n$ is contained in every maximal cyclic subgroup of $G$, and hence $C_n\leq \Cyc(G)$, by Lemma~\ref{l:cyc}$(ii)$. Thus $G/\Cyc(G)$ is a $p$-group and so, by Lemma \ref{l:p-groups}, $\mathcal{E}(G/\Cyc(G))$ is a cograph. Finally, by Corollary~\ref{c:E(G/Cyc)}, $\mathcal{E}(G)$ is a cograph.
\end{proof}

Note that Theorem~\ref{t:Ma24_Thm2} cannot be generalised too much: for example, if $G$ is a  non-nilpotent supersolvable group, then $\mathcal{E}(G)$ may be a chordal graph but not a cograph (see Corollary \ref{supersolvable}).

\subsection{Induced paths and cycles in enhanced power graphs}  We explore in more detail some properties of induced paths and cycles in enhanced power graphs and power graphs. 
We start with an observation about paths. 
\begin{remark}\label{r:cheryl}  Let $G$ be a 
group and let $\gamma=(x_1,x_2,x_3,x_4)$ be a 
$4$-path in  $\mathcal{E}(G)$. If $\set{x_2,x_3}$ is 
an edge of the power graph $\mathcal{P}(G)
$, then at least one of $\set{x_2,x_4}$ or 
$\set{x_1,x_3}$ is an edge in $\mathcal{E}(G)$. In 
particular, $\gamma$ is not induced.
\end{remark}

\begin{proof} Since $\set{x_1,x_2}$ and $\set{x_3,x_4}$ are edges of $\mathcal{E}(G)$, both $\langle x_1, x_2\rangle$ and $\langle x_3, x_4\rangle$ are cyclic. Assume that $\set{x_2,x_3}$ is an edge of $\mathcal{P}(G)$. If $x_2$ is a power of $x_3$, then $\langle x_2,x_4\rangle\leq \langle x_3,x_4\rangle$, which is is cyclic, and thus $\set{x_2,x_4}$ is an edge in $\mathcal{E}(G)$. Similarly, if $x_3$ is a power of $x_2$, then $\langle x_1,x_3\rangle\leq \langle x_1,x_2\rangle$, which is cyclic, and thus $\set{x_1,x_3}$ is an edge in $\mathcal{E}(G)$.
\end{proof}

%\D{D: I worked a lot here. Please do not change, for the moment}
\begin{lemma}\label{l:general} Let $G$ be a  group, let $n\geq 4$, let $X=\{x_1,\dots, x_n\}\subseteq G$, and suppose that either $\gamma=(x_1,x_2,\cdots, x_n)$ is an induced $n$-path of $\mathcal{E}(G)$,  or $\mathcal{C}=(x_1,x_2,\cdots, x_n,x_1)$ is an induced $n$-cycle of $\mathcal{E}(G)$.  Then the following hold:
\begin{itemize}
\item[$(i)$] Let $A,B\subseteq [n]$ be 
distinct subsets  with $|A|=|B|=2$, and let 
$ X_A:=\langle x_i:i\in A\rangle$ and $X_B:=\langle x_i:i\in B\rangle$. If $X_A$ and $X_B$ are cyclic, then $X_A\neq X_B.$

\item[$(ii)$] For distinct $i, j\in [n]$, we have $\langle x_i\rangle\neq \langle x_j\rangle$. 

\item[$(iii)$] $X\cap \mathrm{Cyc}(G)=\varnothing$. In particular,  $x_i\neq 1$, for all $i\in [n]$. 

\item[$(iv)$]  If $X$ is the vertex set of the path $\gamma$, then $X\cap \mathrm{sl}(\mathcal{E}(G))\subseteq \{x_1, x_n\};$
if $X$ is the vertex set of the cycle $\mathcal{C}$, then $X\cap \mathrm{sl}(\mathcal{E}(G))=\varnothing.$

\item[$(v)$] For $1<i<n$ in the case of the path $\gamma$, and any $i\in [n]$ in the case of the cycle $\mathcal{C}$, the element $x_i$   belongs to at least two distinct maximal cyclic subgroups of $G$.

\item[$(vi)$] Every internal edge of $\gamma$, and every edge of $\mathcal{C}$, belongs to $E(\mathcal{E}(G)\setminus \mathcal{P}(G)).$

\item[$(vii)$] Let $\{x,y\}$ be an internal edge of $\gamma$, or any edge of $\mathcal{C}$.  Then $|x|$ does not divide $|y|$ and $|y|$ does not divide $|x|$.
\end{itemize}    
\end{lemma}

\begin{proof}
$(i)$ Suppose to the contrary that $X_A$ and $X_B$ are cyclic and $X_A=X_B$. Assume first that $A\cap B\ne\varnothing$, so $A=\{i,j\}$ and $B=\{i, k\}$ for pairwise distinct $i,j,k\in [n].$ 
Then $\langle x_i, x_j\rangle=X_A=X_B=\langle x_i, x_k\rangle$ and thus $\langle x_i, x_j, x_k\rangle$ is cyclic.
Therefore  $\{x_i, x_j, x_k\}$ induces a $3$-cycle $\nu$ in $\mathcal{E}(G)$. 
However $\nu$ is a subgraph of either $\gamma$ or $\mathcal{C}$, and we have a contradiction. 
Thus $A\cap B=\varnothing$, so $A=\{i,j\}$ and $B=\{k, \ell\}$ for pairwise distinct $i,j,k,\ell\in [n].$ 
Then $\langle x_i, x_j\rangle=X_A=X_B=\langle x_k, x_\ell\rangle$ and thus $\langle x_i, x_j, x_k, x_\ell\rangle$ is cyclic. Therefore $\{x_i, x_j, x_k,x_\ell\}$ induces a complete subgraph  $\nu\cong K_4$ in $\mathcal{E}(G)$, and again we have a contradiction since  $\nu$ is a subgraph of either $\gamma$ or $\mathcal{C}$.

\smallskip
$(ii)$ Assume to the contrary that $\langle x_i\rangle= \langle x_j\rangle$, for some distinct $i, j\in [n]$. Then, interchanging $i$ and $j$ if necessary, we may assume that $\{x_i,x_k\}$ is an edge of $\mathcal{E}(G)$, for some $k\in[n]$ such that $i,j,k$ are pairwise distinct. Thus $\langle x_i,x_k\rangle=\langle x_j,x_k\rangle$ is cyclic.
This contradicts $(i)$, with $A\coloneq\{i,k\}$ and $B\coloneq\{j,k\}$.

\smallskip
$(iii), (iv)$ By Lemma~\ref{l:cyc}, the cycliciser $\mathrm{Cyc}(G)$ is the subset of star vertices of $\mathcal{E}(G)$, and hence part $(iii)$ follows from Lemma \ref{l:path-simp} $(ii)$ and $(iii)$, while part $(iv)$ follows from Lemma \ref{l:path-simp}$(i)$ and $(iii)$.
 
\smallskip
 $(v)$ Let $i$ be as in part $(v)$. Then by part $(iv)$, the vertex $x_i$  is not a simplicial vertex of $\mathcal{E}(G)$. Hence, by Lemma~\ref{l:simplicial},  $x_i$ lies in at least two distinct maximal cyclic subgroups of $G$.

 \smallskip
 $(vi), (vii)$ Let $e=\{x,y\}$ be an internal edge of $\gamma$ or an arbitrary edge of $\mathcal{C}$. Then, if necessary relabeling the vertices in the case $e\in E(\mathcal{C})$, we have $x=x_i,y=x_{i+ 1}$ for some $i$ satisfying $2\leq\ i\leq n-2.$
 It follows that $\gamma':=( x_{i-1},x_i,x_{i+1},x_{i+2})$ is  path of $\mathcal{E}(G)$ and, since $\gamma'$ is a subgraph of $\gamma$ or $\mathcal{C}$, it must be an induced $4$-path of $\mathcal{E}(G)$.
 Hence, by Remark~\ref{r:cheryl}, $\{x_i,x_{i+ 1}\}$ is not an edge of $\mathcal{P}(G)$, proving part $(vi)$. Since  $\{x_i,x_{i+ 1}\}\in E(\mathcal{E}(G))$, the group $\langle x_i, x_{i+ 1}\rangle$ is cyclic of order $m$, say, and so has a unique subgroup of order $m'$ for each divisor $m'$ of $m$. In particular, if either $|x_i|$ divides $|x_{i+1}|$, or $|x_{i+1}|$ divides $|x_{i}|$, then $\{x_i,x_{i+ 1}\}$ is an edge of the power graph $\mathcal{P}(G)$, which contradicts part $(vi)$. This proves part $(vii)$.
\end{proof}

We give two applications of Lemma~\ref{l:general}.
The first reduces the problem of deciding whether $\mathcal{E}(G)$ has the property  of being chordal, $C_n$-free for some $n\geq 4$, or a cograph, to the same decision problem for the vertex deleted graph  $\mathcal{E}(G)- U$, for a certain subset $U$ of vertices.

\begin{lemma}\label{l:cut-cyc-simp}
Let $G$ be a  group.
\begin{enumerate}
\item[$(i)$]
Let $U\subseteq \mathrm{Cyc}(G).$ Then 
$\mathcal{E}(G)-U$ is a cograph if and only if 
$\mathcal{E}(G)$ is a cograph.
\item[$(ii)$] Let $U\subseteq 
\mathrm{sl}(\mathcal{E}(G))\cup\mathrm{Cyc}(G)$ and let $n\geq 4.$ Then  
$\mathcal{E}(G)- U$ is $C_n$-free, or chordal, if and only if $\mathcal{E}(G)$ is $C_n$-free, or  chordal, respectively.
\end{enumerate}
\end{lemma}
    
\begin{proof} 
$(i)$  We show that the set of induced $4$-paths in $\mathcal{E}(G)$ and in $\mathcal{E}(G)- U$ are the same.
Assume that  $\gamma$ is an induced $4$-path in $\mathcal{E}(G)$. Then by Lemma~\ref{l:general}$(iii)$, none of the vertices of $\gamma$ lies in $\mathrm{Cyc}(G)$, and hence $\gamma$ is an induced $4$-path in 
$\mathcal{E}(G)- U$. Conversely suppose that $\gamma$ is an induced $4$-path in $\mathcal{E}(G)- U$. Then, since each edge of $\mathcal{E}(G)$ between two vertices of $\gamma$ is also an edge of $\mathcal{E}(G)- U$, the subgraph $\gamma$ is also an induced $4$-path of $\mathcal{E}(G)$. 

$(ii)$  We show that, for each $n\geq4$, the set of induced $n$-cycles in $\mathcal{E}(G)$ and in $\mathcal{E}(G)- U$ are the same. 
Assume that $\mathcal{C}$ is an induced $n$-cycle in $\mathcal{E}(G)$. Then by Lemma~\ref{l:general}$(iii)$ and $(iv)$, none of the vertices of $\mathcal{C}$ lies in $\mathrm{sl}(\mathcal{E}(G))\cup\mathrm{Cyc}(G)$, and hence $\mathcal{C}$ is an induced $n$-cycle in 
$\mathcal{E}(G)- U$. Conversely, if $\mathcal{C}$ is an induced $n$-cycle in  $\mathcal{E}(G)- U$, then the same argument used in part $(i)$ proves  that $\mathcal{C}$ is an induced $n$-cycle  of $\mathcal{E}(G)$.
\end{proof}

Lemma~\ref{l:general} also provides simpler proofs of some known results about EPPO-groups. For instance, it leads to a clearer proof of the result \cite[Proposition 2]{Ma24}  of Ma et al.  that, for an EPPO group $G$,   $\mathcal{E}(G)$ is a cograph and is chordal.

\begin{prop} \label{known}
Let $G$ be an {\rm EPPO}-group. Then 
\begin{enumerate}
    \item[$(i)$]  {\rm \cite[Proposition 2]{Ma24}}\quad  $\mathcal{E}(G)$ is a cograph and is chordal.
    \item[$(ii)$] $\mathcal{P}(G)$ is a cograph and is chordal.
    \end{enumerate}
\end{prop}

\begin{proof}
$(i)$ By \cite[Theorem 28]{alip},
$\mathcal{E}(G)\setminus 
\mathcal{P}(G)$ is a null graph with vertex set $G$ and no edges. 
Thus, by Lemma~\ref{l:general}\,$(vi)$, $\mathcal{E}
(G)$ contains no induced $4$-path and no $n$-cycle for any $n\geq4$. Thus  $\mathcal{E}(G)$ is  
a cograph and is chordal. %We note also that we could have used Lemma~\ref{l:general}\,$(vi)$ to deduce that $\mathcal{E}(G)$ is  
%a cograph and then applied Theorem A to obtain that it is chordal. \D{I am happy with what we have. I think that, for the sake of brevity, there is no need to add the alternative road by Theorem A. What do you think?}

$(ii)$ This  follows from part $(i)$, on recalling that for EPPO-groups the power graph and the enhanced power graph coincide.
\end{proof}
Note that Proposition~\ref{known} improves an earlier result of Doostabadi et al. \cite[Corollary 3.2]{Doo} stating  that  the power graph of an EPPO-group is $C_4$-free.

\subsection{$C_4$-free enhanced power graphs.}
Here we give  a group theoretic criterion for an enhanced power graph to have an induced $4$-cycle  (see Proposition \ref{p:AB}).
We start by recalling a useful result from \cite{alip}.

\begin{lemma}\cite[Lemma 32]{alip}\label{alipur1} Let $G$ be a group and let $x,y,z\in G.$ If $(x,y,z,x)$ is a $3$-cycle in $\mathcal{E}(G)$, then $\langle x,y,z\rangle$ is cyclic.
\end{lemma}

\begin{lemma}\label{l:gencon1} 
Let $G$ be a  group. 
Assume that $\mathcal{E}(G)$ is not chordal and let $n$ be minimal such that $n\geq 4$ and $\mathcal{E}(G)$ contains an induced $n$-cycle. Then there is such an induced $n$-cycle 
$$\mathcal{C}=(a_1,a_2,\ldots,a_n,a_1)$$
with the following properties:
\begin{enumerate}
\item[$(1)$] each $a_i\in G\setminus \mathrm{sl}(\mathcal{E}(G))$; %\D{can we subtract also the cycliciser? As in Lemma \ref{l:cut-cyc-simp}, I mean}
\item[$(2)$] each $a_i$ has prime power order;
\item[$(3)$] $\gcd(|a_i|,|a_{i+1}|)=1$ for all $i\in[n]$, where $a_{n+1}\coloneq a_1$.
\end{enumerate}
\end{lemma}
%\D{D: if useful for us, we can prove the same result for the case of paths, with a similar proof}

\begin{proof} 
% Let $\delta=(a_1,a_2,\ldots,a_n,a_1)$ be an induced $n$-cycle of $\mathcal{E}(G)$. 
Assume that some vertex of $\mathcal{C}$ has non-prime-power 
order. Then without loss of generality, $|a_1|=m$ with $m$ not a prime power. Let $\pi(m)=\{p_1,\dots, p_t\}$ be the set of primes dividing $m$, where $|\pi(m)|=t\geq2$, and for each $p\in \pi(m)$ write $m=m_p\ell_p$ where $m_p$ is a $p$-power and $\ell_p$ is coprime to $p$.  
% let  $m_p=p^{\alpha_p}$ be the $p$-part of $m$, that is, the maximal power of $p$ dividing $m$ and set $l_p=m/m_p$. 
Then $(a_1)^{\ell_p}$ is nontrivial and has order $m_p$. 
Since $\{a_1,a_2\}$ and $\{a_1,a_n\}$ are edges of $\mathcal{E}(G)$, both $\seq{a_1,a_2}$ and $\seq{a_1,a_n}$ are cyclic, and hence also $\seq{(a_1)^{\ell_p},a_2}$ and $\seq{(a_1)^{\ell_p},a_n}$ are cyclic. Thus  $\mathcal{C}_1:=\big((a_1)^{\ell_p},a_2,\ldots,a_n,(a_1)^{\ell_p}\big)$ is an $n$-cycle in $\mathcal{E}(G)$.

We claim that $\seq{(a_1)^{\ell_p},a_3}$ is not cyclic for some $p\in\pi(m)$. Assume to the contrary that, for each $p\in\pi(m)$, $\seq{(a_1)^{\ell_p},a_3}$ is cyclic. 
% and hence $((a_1)^{\ell_p},a_2, a_3, (a_1)^{\ell_p})$ is a $3$-cycle.  
We apply Lemma \ref{alipur1} recursively, for  $i\in\{2,\dots,t\}$, with $z=a_3$, $y=(a_1)^{\ell_{p_i}}$, and 
$x=\prod_{j<i}(a_1)^{\ell_{p_j}}$, noting that, for each application, we have that $(x,y,z,x)$ is a $3$-cycle in $\mathcal{E}(G)$. We 
conclude that  $\seq{a_3,(a_1)^{\ell_{p_1}},\dots, 
(a_1)^{\ell_{p_t}}}$ is cyclic. However, 
$\seq{(a_1)^{\ell_{p_1}}, \dots,
(a_1)^{\ell_{p_t}}}=\seq{a_1}$. So this implies that 
$\seq{a_3,a_1}$ is cyclic and hence $\{a_1,a_3\}$ is an 
edge of $\mathcal{E}(G)$,  which is a contradiction. 
Thus the claim is proved. 

Consider now a prime $p\in\pi(m)$ such that $\seq{(a_1)^{\ell_{p}},a_3}$ is not cyclic. Then $(a_1)^{\ell_{p}}$ is not adjacent to $a_3$, while $(a_1)^{\ell_{p}}$ is adjacent to $a_n$ (since $\seq{(a_1)^{\ell_{p}}, a_n}\leq \seq{a_1, a_n}$, which is cyclic). 
Let $k$ be the least integer such that $3<k\leq n$ and 
$(a_1)^{\ell_{p}}$ is adjacent to $a_k$. Now since $\mathcal{C}$ is an induced $n$-cycle, $\set{a_i,a_{j}}$ is not an edge of $\mathcal{E}(G)$ whenever $i,j\neq 1$ and $|i-j|>2$. Thus the sequence 
$\mathcal{C}_2:=\big((a_1)^{\ell_p},a_2,\ldots,a_k,(a_1)^{\ell_p}\big)$
is an induced $k$-cycle in $\mathcal{E}(G)$. Since $\mathcal{C}_2$ is a subgraph of $\mathcal{C}$, this implies that $k=n$. Thus 
$\mathcal{C}_2:=\big((a_1)^{\ell_p},a_2,\ldots,a_n,(a_1)^{\ell_p}\big)$ is an induced $n$-cycle in $\mathcal{E}(G)$, and the number of its vertices having prime power order is greater than that number for $\mathcal{C}$. Repeating this argument for any of $a_2,\ldots ,a_n$ which have  non-prime-power order, we obtain an induced $n$-cycle 
such that all its vertices have prime power order. We assume from now on that each $|a_i|$ is a prime power.

By Lemma \ref{l:general} $(iv)$ each $a_i\in G\setminus \mathrm{sl}(\mathcal{E}(G))$.
Finally we show that $\gcd(|a_i|,|a_{i+1}|)=1$ for every $i\in[n]$. 
Assume that this is false and that  $\gcd(|a_i|,|a_{i+1}|)\neq 1$. Since both $|a_i|$ and $|a_{i+1}|$ are prime powers, this implies that $|a_i|$ and $|a_{i+1}|$ are powers of the same prime, say $p$, so $\seq{a_i,a_{i+1}}$ is a cyclic $p$-group. Therefore either $\seq{a_i}\leq \seq{a_{i+1}}$ or $\seq{a_{i+1}}\leq \seq{a_i}$. The first case implies that $\seq{a_i,a_{i+2}}\leq \seq{a_{i+1},a_{i+2}}$, which is cyclic, and thus $\set{a_{i},a_{i+2}}$ is an edge, and this is a contradiction. Similarly, the second case implies that  
$\seq{a_{i+1},a_{i-1}}\leq \seq{a_i,a_{i-1}}$ and therefore $\seq{a_{i+1},a_{i-1}}$ is cyclic, so that $\set{a_{i+1},a_{i-1}}$ is an edge, which again is a contradiction.
\end{proof}

Now we prove the most general form of our criterion for an induced $C_4$ subgraph.

\begin{prop}\label{p:AB}
Let $G$ be a  group. Then the following are equivalent.
\begin{enumerate}
    \item[$(a)$]  $\mathcal{E}(G)$ has an 
induced $4$-cycle;
    \item[$(b)$] $G$ has a subgroup $H$ of the form $H=AB$, where $A,B\leq G$ satisfy the following conditions:
    \begin{enumerate}
        \item[$(i)$] $A=\seq{a_1,a_2}$ and $B=\seq{b_1,b_2}$ are both non-cyclic;
        \item[$(ii)$] $A$ and $B$ centralise each other;
        \item[$(iii)$] for $i,j\in[2]$, $|a_i|$ and $|b_j|$ are prime powers, and $\gcd(|a_i|,|b_j|)=1$.
    \end{enumerate}
\end{enumerate}
\end{prop}

\begin{proof}
Assume first that condition $(b)$ holds for a subgroup $H=AB$. Since $A$ and $B$ centralise 
each other and  $\gcd(|a_i|,|b_j|)=1$ for all $i,j\in[2] $, the sequence 
$$
\mathcal{C}=(a_1,b_1,a_2,b_2,a_1)
$$
is a $4$-cycle in $\mathcal{E}(G)$. Moreover, since $A$ and $B$ are not 
cyclic, $\mathcal{C}$ is an induced $4$-cycle, and hence condition $(a)$ holds.

%\Fr{[To be shortened in accordance with the new Lemma \ref{l:gencon1}.]}\\
Conversely, assume that condition $(a)$ holds, and let $\mathcal{C}$ be an induced 
$4$-cycle in  $\mathcal{E}(G)$. By Lemma 
\ref{l:gencon1} we may assume that 
$$
\mathcal{C}=(a_1,b_1,a_2,b_2,a_1)
$$
where $|a_i|$ and $|b_j|$ are prime powers for each $i,j\in[2]$ and such that $\gcd(|a_i|,|b_j|)=1$ for every $i,j\in[2]$. Set $A=\seq{a_1,a_2}$, $B=\seq{b_1,b_2}$, and $H=\seq{A, B}$. Since $\set{a_1,a_2}$ and $\set{b_1,b_2}$ are not edges in $\mathcal{E}(G)$, both $A$ and $B$ are not cyclic. Moreover,  for each $i,j\in [2]$, $\set{a_i,b_j}$ is an edge, so the subgroup $\seq{a_i,b_j}$ is cyclic. This implies that the subgroups $A$ and $B$ centralise each other, so that $H=AB$. 
\end{proof}

A consequence is the following corollary.

\begin{corollary}\label{c:AB_maximal}
If $G$ is a finite nonabelian simple group, then 
$\mathcal{E}(G)$ is $C_4$-free if and only, for every 
maximal subgroup $M$ of $G$, 
$\mathcal{E}(M)$ is $C_4$-free.   
\end{corollary}

\begin{proof}
One implication is immediate by Corollary  
\ref{c:Ma24_Lem1}. Assume therefore that for 
every maximal subgroup $M$ of $G$ the enhanced 
power graph $\mathcal{E}(M)$ is $C_4$-free 
and, suppose to the contrary that $\mathcal{E}(G)$ is not $C_4$-free. 
Then by Proposition ~\ref{p:AB}, 
$G$ 
has a subgroup $H=AB$ as described in part $(b)$ of  Proposition \ref{p:AB}. Since $A$ and $B$ are non-cyclic, each generator $a_i\ne 1$ and $b_j\ne 1$, and since $A$ and $B$  centralise each other, in particular $b_1$ centralises $A$.
Then since $G$ is a finite nonabelian  simple group, $Z(G)=1$ and hence $1\neq A\neq G$. Also, since $A\unlhd H$, it follows that $H\neq G$. In particular there exists a maximal subgroup $M$ of $G$ containing $H$, and hence $\mathcal{E}(M)$ is not $C_4$-free by Lemma \ref{l:Ma24_Lem1}. This contradicts our assumption, and hence completes the proof. 
\end{proof}

The following explicit examples of groups $H$ satisfying the conditions of Proposition~\ref{p:AB} will prove useful for our proof of Theorem B.

\begin{corollary}\label{c:AB_application}
    Suppose that a  group $G$ has a subgroup $H$, with $H$ one of the following groups:
    \begin{enumerate}
        \item[$(a)$]  $C_r^2\times C_s^2$, for $r,s$ distinct primes;
        \item[$(b)$]  $A_4\times S_3$ or $A_4\times C_2^2$;
        \item[$(c)$]  $\SL_2(p)\times \SL_2(p)$ for an odd prime $p$, or $\SL_2(3)\times Q_8$;
        \item[$(d)$]  $\SL_2(p)\circ \SL_2(p)$ for an odd prime $p$, or $\SL_2(3)\circ Q_8$.
        \item[$(e)$]  $\PSL_2(q)\times \PSL_2(q)=\SL_2(q)\times\SL_2(q)$ for  $q=2^f\geq4$.
\end{enumerate}
Then $\mathcal{E}(G)$ is not $C_4$-free. In particular, $\mathcal{E}(G)$ is not chordal and not a cograph. 
\end{corollary}

% \C{If needed later we can strengthen b) to $A\times B$, where $A\in\{C_2^2, S_3\}$ and $B\in\{ A_4, \SL_2(3), C_3^2\}$. I found some gaps in July for $U_5(2), U_4(2)$ where the extra cases would help.}
% \D{I worked on those groups. I believe that they are fine now, so we do not need to strengthen. I suggest to look before to the parts where you saw mistakes in July. We have worked a lot on the file since then... can you say to me and Francesco which mathematical mistakes you have found?}
\begin{proof}
In all cases it is sufficient to prove that $\mathcal{E}(H)$ contains an induced $4$-cycle.

$(a)$ In this case $H$ is nilpotent and the assertion follows from Theorem \ref{t:Ma24_Thm2}. 

$(b)$ For both groups we take 
$A=A_4 =\langle a_1, a_2\rangle$ with $|a_i|=3$, and 
$B=\langle b_1, b_2\rangle$  with $|b_i|=2$ (and $B=S_3$ 
or $C_2^2$) so that $H=A\times B$. All 
conditions of Proposition ~\ref{p:AB}$(b)$ hold for these 
generators, and the assertion follows  from that result.

$(c)$-$(d)$ In both cases we let 
$A=\langle a_1, a_2\rangle = \SL_2(p)$
where each $a_i$ is an element of order $p$  in the 
first component of $H$, and we let $B=\langle b_1, 
b_2\rangle$ where each $b_i$ is an element of order $4$ in the second component of $H$. Note that $B\cong Q_8$ if 
$p=3$, and $B\cong\SL_2(p)$ if $p>3$. In all cases  
both subgroups $A$ and $B$ are non-cyclic, and $A$ and $B$ centralise each other. Thus $H=AB$ satisfies the conditions 
of Proposition ~\ref{p:AB} $(b)$ and $\mathcal{E}(H)$ contains an induced $4$-cycle.

$(e)$ Since $q=2^f\geq4$, each of $q+1, q-1$ has an odd prime divisor $r, s$, say, and $r\ne s$. Also  $H=A\times B$ where 
$A=\seq{a_1, a_2}= \SL_2(q)$  with each 
$|a_i|=r$ and  $B=\langle b_1, 
b_2\rangle=\SL_2(q)$ with each $|b_i|=s$. Then all the 
conditions of Proposition \ref{p:AB}$(b)$ hold, and the assertion follows from that result.
\end{proof}

\section{enhanced power graphs which are block graphs}\label{sect:block}
 
We begin by recalling some classic terminology from graph theory, as used in Harary's paper \cite{harary}.
Let $\Gamma=(V,E)$ be a graph. %Given $x\in V$, the subgraph $\Gamma-x\coloneq \Gamma[V\setminus\{x\}]$ is called the \emph{vertex-deleted subgraph} of $\Gamma$, with respect to the vertex $x.$ The graph  
Then $\Gamma$ is said to be {\bf $2$-connected} if $|V|\geq 2$, $\Gamma$ is connected, and for each $x\in V$ the subgraph $\Gamma-x$ is connected. Now let $\Gamma=(V,E)$ be a connected graph, with $|V|\geq 2.$ 
Then a $2$-connected subgraph $B$  of $\Gamma$ is called a {\bf block} of $\Gamma$ if $B$ is $2$-connected maximal, that is, there exists no $2$-connected subgraph $\Delta$ of $\Gamma$ such that $B\subsetneq \Delta$. Note that a block of $\Gamma$ is necessarily an induced subgraph of $\Gamma$ and contains at least two vertices. For instance, for a $3$-path $\Gamma=(a,b,c)\cong P_3$, the blocks are  $B_1=\Gamma[\{a,b\}]$ and $B_2=\Gamma[\{b,c\}]$, and 
for a $4$-path $\Gamma=(a,b,c,d)\cong P_4$, the blocks are $B_1=\Gamma[\{a,b\}]$,  $B_2=\Gamma[\{b,c\}]$,  and $B_3=\Gamma[\{c,d\}].$

The notion of a block extends naturally  to the case where $\Gamma$ is not connected: we define the blocks of $\Gamma$ as the blocks of those connected components of $\Gamma$ which contain at least two vertices. In particular, the union of the vertex sets of the blocks of $\Gamma$ is not necessarily equal to the whole vertex set of $\Gamma$. 
However, if $\Gamma=(V,E)$ is connected and $|V|\geq 2$, then the union of the vertex sets of the blocks is equal to $V$ and distinct blocks intersect in at most one vertex, \cite[Theorem B]{harary}.

A graph $\Gamma$ is called a {\bf block graph} 
if  every block of $\Gamma$ is a complete graph, see \cite[Theorems 1 and 2]{harary}.
Note that by adding some isolated vertices to a block graph one obtains a new block graph. We observe the following:
\begin{itemize}
    \item \emph{A block graph is chordal}: Indeed, suppose that $\Gamma$ is a block graph and that  $\mathcal{C}\cong C_n$ is a subgraph, with $n\geq 4.$  Then $\mathcal{C}$ is $2$-connected and thus, by definition of a block, its vertices are included in a single block. However, since a block of $\Gamma$ is a complete subgraph, each pair of distinct vertices of $\mathcal{C}$ is adjacent in $\Gamma$ and thus $\mathcal{C}$ is not induced, so $\Gamma$ is chordal.
    \item  \emph{A block graph is not necessarily a cograph}: for example, $\Gamma=(a,b,c,d)\cong P_4$ is not a cograph, but we have seen that each of its blocks is a complete graph on two vertices, and hence $\Gamma$ is a block graph.
\end{itemize}

For the purposes of our paper, it is useful to give the following two characterisations of block graphs. Though they seem to be well-known, we could not find a precise reference in the literature, so for the sake of completeness we provide a proof. Recall that the graph  obtained from $K_4$ by removing one of its edges is called a {\bf diamond}.  

\begin{prop}\label{t:char-block}
Let $\Gamma$ be a graph. Then the following are equivalent.
\begin{itemize}
    \item[$(a)$] $\Gamma$ is a block graph;
    \item[$(b)$] $\Gamma$ is chordal and two distinct maximal cliques share at most one vertex;
    \item[$(c)$] $\Gamma$ is chordal and diamond-free.
\end{itemize}
    
\end{prop}
\begin{proof}$(a)\Rightarrow (b).\quad$
Assume that $\Gamma$ is a block graph. By our discussion above, $\Gamma$ is chordal. If $\Gamma$ is a complete graph then condition $(b)$ holds. So suppose that $\Gamma$ is not complete and let 
$X\neq Y\subseteq V$  be distinct maximal cliques of $\Gamma.$
Note that $Y\subsetneq X$ by the maximality of $Y$, and hence $X\cup Y\supsetneq X$.  
We need to show that $|X\cap Y|\leqslant
1$. Suppose to the contrary that $|X\cap Y|\geq 2$. Then $\Gamma[X\cup Y]$ is $2$-connected and thus is contained in a block $B$ of $\Gamma$. Since $\Gamma$ is a block graph, $B$ is complete with  $B=\Gamma[Z]$ for some $Z\subseteq V$ such that  $Z\supseteq X\cup Y.$ Since $X\cup Y\supsetneq X$, this implies that $Z\supsetneq X$, contradicting the maximality of the clique $X$. 

$(b)\Rightarrow (c).\quad$ Let $\Gamma=(V,E)$ be a graph satisfying condition $(b).$ Then $\Gamma$ is chordal and we need to show that $\Gamma$ is diamond-free. Assume to the contrary that the induced subgraph $\mathcal{D}\leq \Gamma$ is a diamond with vertex set $D=\{a,b,c,d\}\subseteq V$ and edge set given by all the $2$-subsets of $D$ except $\{a,c\}.$ 
Then $\Gamma$ contains the  cliques $K_a\coloneq\{a,b,d\}$ and $K_c\coloneq\{b,c,d\}$. Let $\hat K_a$ be a maximal clique containing $ K_a$ and $\hat K_c$ a maximal clique containing $ K_c$. Then $c\notin \hat K_a$ because $\{a,c\}\notin E$ while $c\in \hat K_c$. Thus $\hat K_a\neq \hat K_c$  and $\hat K_a\cap \hat K_c\supseteq \{a,b,d\}\cap \{b,c,d\}=\{b,d\}.$ Thus the maximal cliques $\hat K_a$ and $\hat K_c$ share at least two vertices, which is a contradiction.

$(c)\Rightarrow (a).\quad$ This follows  from \cite[Proposition 1]{ba}.
\end{proof}

Note that, in Proposition \ref{t:char-block}, the property of being chordal is essential. For example, in the graph $\Gamma=([8],E)$ with $E$ given by the edges in the $8$-cycle $(1,2,\dots,8,1)$ together with those in the $4$-cycle $(2,4,6,8,2)$, 
%$$E\coloneq \{\{1,2\}, \{2,3\}, \{3,4\}, \{4,1\},\{1,5\},\{1,6\}, \{2,6\}, \{2,7\},\{3,7\},\{3,8\}, \{4,5\},\{4,8\}\}.$$ 
any two distinct maximal cliques share at most one vertex, but the only block is the graph $\Gamma$ itself, which is not a clique.
For an enhanced power graph, the property of being a block graph can be instead 
%\mathcal{E}\coloneq \{\mathcal{E}(G): G \hbox{ is a  group}\}
characterised without reference to the property of being chordal (see Theorem \ref{p:partition}). Here the class of groups admitting a nontrivial partition comes into play. 
A {\bf nontrivial partition} of a non-identity group $G$ is a set $\mathcal{P}=\{H_1, H_2,\dots, H_n\}$ of $n\geq 2$ proper nontrivial subgroups of $G$ such that $G=\bigcup_{i=1}^n H_i$ and  
$H_i\cap H_j=1$ for distinct $i, j\in [n].$
For a broad survey about partitions of finite groups, the reader can refer to \cite{Zappa}.

For the following remark, recall the definition of $\mathcal{M}(G)$ given in \eqref{e:M}.
%For the next results, recall the definitions of $\mathcal{M}(G)$ in \eqref{e:M}, and {\rm sl}$(\mathcal{E}(G))$ just before Remark~\ref{r:star-simplicial}.

\begin{remark}\label{r:partn}
{\rm Suppose that $G$ is a  non-cyclic group with nontrivial partition 
$\mathcal{P}=\{H_1, H_2,\dots, H_n\}$. We note the following.
\begin{enumerate}
\item[$(i)$] For $i\in [n]$, if $H_i$ is cyclic, then $H_i\in\mathcal{M}(G)$. 
\item[$(ii)$] If all the $H_i$ are cyclic then $\mathcal{P}=\mathcal{M}(G)$. 
\end{enumerate}
}
\end{remark}

\begin{proof}
$(i)$ Suppose that $H_i\in \mathcal{P}$ is cyclic and   $H_i<\seq{z}$ for some $\seq{z}\in\mathcal{M}(G)$. As $\mathcal{P}$ is a partition $z\in H_j$ for some $j\ne i$. This however implies  that $H_j\cap H_i=H_i\neq 1$, which contradicts the fact that $\mathcal{P}$ is a partition.

$(ii)$ By part $(i)$, $\mathcal{P}\subseteq \mathcal{M}(G)$. Conversely, for each $\seq{z}\in\mathcal{M}(G)$, the generator $z$ lies in some $H_i$ and since $H_i$ is cyclic it follows that $H_i=\seq{z}$, so $\seq{z}\in\mathcal{P}$.
\end{proof}
%There is a further interesting class of groups coming into play when characterising the groups whose enhanced power graph is a block graph, the class of C-tidy groups introduced in \cite{bai} which we now recall.
%Let $G$ be a group. For $x\in G$ consider the set $\Cyc(x)\coloneq \{y\in G: \seq{y,x}\ \hbox{is cyclic}\}$ introduced in \cite{Wepsic0}. Clearly, we have $\Cyc(x)=N[x]$ and $\Cyc(G)=\cap_{x\in G}\Cyc(x)$. A group $G$ is called a {\bf C-tidy group} if, for every $x\in G\setminus \Cyc(G)$, $\Cyc(x)$ is a cyclic subgroup of $G$.
There are two further interesting class of groups coming into play when characterising the groups whose enhanced power graph is a block graph, the class of C-tidy groups introduced in \cite{bai} and the class of cyclic-transitive groups introduced in \cite{Lewis}. We recall the definition of those classes.
Let $G$ be a group. For $x\in G$ consider the set $\Cyc(x)\coloneq \{y\in G: \seq{y,x}\ \hbox{is cyclic}\}$ introduced in \cite{Wepsic0}. Clearly, we have $\Cyc(x)=N[x]$ and $\Cyc(G)=\cap_{x\in G}\Cyc(x)$. The group $G$ is called:
\begin{itemize}
    \item a {\bf C-tidy group} if, for every $x\in G\setminus \Cyc(G)$, $\Cyc(x)$ is a cyclic;  
 \item a {\bf cyclic-transitive group} if, for every $x,y,z\in G\setminus\{1\},$ $\seq{x,y}$ and $\seq{y,z}$ cyclic implies $\seq{x,z}$ cyclic.
\end{itemize}

For the final main result of this section, recall also the definition of {\rm sl}$(\mathcal{E}(G))$ given in Definition \ref{d:simplicial}.
\begin{theorem}\label{p:partition} Let $G$ be a  group. Then the following are equivalent.
\begin{itemize}
    \item [$(a)$] $\mathcal{E}(G)$ is diamond-free;
    \item [$(b)$] every nontrivial cyclic  subgroup of $G$ is contained in a unique maximal cyclic subgroup;
    \item [$(c)$] {\rm sl}$(\mathcal{E}(G))\supseteq G\setminus\{1\}$;
    \item [$(d)$] $G$ is cyclic or $G$
admits a nontrivial partition consisting of cyclic subgroups;
    \item [$(e)$]  $\mathcal{E}(G)$ is a block graph and a cograph in which 
any two maximal cliques meet in the identity 
element of $G$;
    \item [$(f)$] $\mathcal{E}(G)$ is a block graph;
    \item[$(g)$] $G$ is cyclic or $G$ is a {\rm C}-tidy group with $\Cyc(G)=1$;
    \item[$(h)$] $G$ is cyclic-transitive.
\end{itemize}
\end{theorem}
\begin{proof} 
  Throughout the proof we write  $\Gamma\coloneq \mathcal{E}(G)$.
 
$(a)\Rightarrow (b).\quad$ We show that the negation of $(b)$ implies the negation of $(a).$  Suppose that some $x\in G\setminus\{1\}$ is contained in two distinct maximal cyclic subgroups $\seq{a}, \seq{b}\in \mathcal{M}(G)$. If $\{a,b\}$ were an edge of $\Gamma$ then $\seq{a,b}$ would be a cyclic subgroup properly containing at least one of $\seq{a}$ or $\seq{b}$, contradicting the definition of $\mathcal{M}(G)$. Thus  
$\{a,b\}\notin E(\Gamma)$, and hence $\Gamma[1,a,b,x]$ is a diamond.

$(b)\Rightarrow (c).\quad$ This follows from Lemma \ref{l:simplicial} (part $(b)\Rightarrow (a)$).

$(c)\Rightarrow (d).\quad$ Assume that {\rm sl}$(\mathcal{E}(G))\supseteq G\setminus\{1\}$, and  first suppose that ${\rm sl}(G)= G$. Then, since $1$ is simplicial, the closed neighbourhood  
$N_\Gamma[1]=G$ induces a complete graph. Thus $\Gamma$ is complete and $G$ is cyclic, by Lemma~\ref{l:cyc}$(iii)$, so $(d)$ holds. Thus we may assume that sl$(\mathcal{E}(G))=G\setminus \{1\}.$ Then $\Gamma = \Gamma[N_\Gamma[1]]$ is not complete and so $G$ is not cyclic by Lemma~\ref{l:cyc}$(iii)$. We claim that $ \mathcal{M}(G)$ is a nontrivial partition of $G$. Now $|\mathcal{M}(G)|\geq 2$ since $G$ is not cyclic, and clearly $G=\bigcup_{H\in \mathcal{M}(G)} H$. Assume to the contrary that $\seq{a}\cap \seq{b}\neq 1$ for some distinct $\seq{a},\seq{b}\in \mathcal{M}(G)$. Then we have  $\seq{x}=\seq{a}\cap \seq{b}\neq \seq{a}, \seq{b}$, for some $x\in G\setminus \{1\}=\mathrm{sl}(\mathcal{E}(G))$, so $a$ and $b$ are distinct neighbors of $x$ and must form an edge of $\Gamma$ since $x$ is simplicial. Thus $\seq{a,b}$ is cyclic and properly contains $\seq{a}$, which is a contradiction.

    $(d)\Rightarrow (e).\quad$ If $G$ is cyclic then $\Gamma$ is complete, by Lemma~\ref{l:cyc}$(iii)$, and thus $\Gamma$ is a block graph and a cograph and part $(e)$ holds. Assume now that $G$ admits a nontrivial partition $\mathcal{P}$ by cyclic subgroups. Then, by Remark~\ref{r:partn}$(ii)$, $\mathcal{P}=\mathcal{M}(G).$
By Lemma \ref{alipur2}, the set of maximal cliques of $\Gamma$ coincides with $\mathcal{M}(G)$. Thus two distinct maximal cliques of $\Gamma$ share exactly one vertex, namely the identity element of $G$, and $\Gamma\setminus\set{1}$ 
is a disjoint union of complete graphs. As a consequence, $\Gamma\setminus\set{1}$ is a cograph
 and so, by Lemma \ref{l:cut-cyc-simp}$(i)$, $\Gamma$ is also a cograph. By Theorem A, $\Gamma$ is also chordal and hence, by Proposition \ref{t:char-block} parts $(a)$ and $(b)$, $\Gamma$ is a block graph.
\smallskip

$(e)\Rightarrow (f).\quad$ This implication is immediate.
\smallskip

$(f)\Rightarrow (a).\quad$ This follows from Proposition \ref{t:char-block} $(a)$ and $(c)$.
\smallskip

$(d)\Longleftrightarrow (g).\quad$ This is the content of \cite[Proposition 2.3]{bai}.
\smallskip

$(d)\Longleftrightarrow (h).\quad$ This is the content of \cite[Theorem 2.2]{Lewis} once observed that, for finite groups, a locally cyclic group is just a cyclic group.
\end{proof}
Lewis and Imperatore in \cite[Theorems 3.2, 4.2]{Lewis} give the complete list of the
cyclic-transitive groups; Baishya in \cite[Theorem 3.13]{bai} gives, three years later, 
the complete list of the C-tidy groups with trivial cycliciser. Of course, by the equivalence of $(g)$ and $(h)$ in Theorem \ref{p:partition}, those two lists coincide, apart the cyclic groups. However, that fact 
is not recognised in \cite{bai}.

For the purposes of our paper, it is remarkable to observe that by those lists and 
by Theorem \ref{p:partition}, we get a complete classification of the groups having enhanced power graph a block graph.
In particular, we are going to use in the following sections the fact that the enhanced power graph of
a dihedral group, and of a Frobenius group  with cyclic kernel and cyclic complement, is a cograph.
%\D{WARNING: It is possible that \cite[Theorem 3.13]{bai} contains a mistake. If you like I can tell you why, later. For this reason we could avoid to tell anything about that theorem. What we need is just the fact about Frobenius and dihedral groups and we have that from part $(d)$.}

Note finally, as another consequence of Theorem \ref{p:partition},  the curious fact that
 for an enhanced power graph, 
being a block graph is equivalent to being diamond-free.

\section{enhanced power graphs of finite nonabelian simple groups}

In this last section we determine the finite nonabelian 
simple groups $G$ for which $\mathcal{E}(G)$ is a cograph, and we obtain partial information about simple groups $G$ for which $\mathcal{E}(G)$ is $C_4$-free. 
%\D{and those whose enhanced power graph is $C_4$-free}.
%respectively it is a chordal graph which is not a cograph.
This is the content of Theorem B.  Our proof relies on the classification of  finite nonabelian  simple groups (CFSG). 

First we handle the group $A_7$: we show in Example~\ref{r:C_3:D_8} that $\mathcal{E}(A_7)$ is not a cograph by exhibiting a supersolvable subgroup of $A_7$  whose enhanced power graph is not a cograph, and then we give in Proposition~\ref{p:A7}  a new  proof of \cite[Proposition 5]{Ma24} that is a simple application of several general results established in Section~\ref{s:ccprops}. 

\begin{example}\label{r:C_3:D_8}
{\rm Let $a\coloneq (567), x\coloneq (1234)(56), y\coloneq (12)(34) \in A_7$ and let 
$K\coloneq \langle a,x,y\rangle\leq A_7.$ Since $a^3=x^4=y^2=1, a^x=a^{-1}, a^y=a, x^y=x^{-1},$ it is clear that 
$K=\seq{a}\rtimes \seq{x,y}\cong C_3:D_8$ is a supersolvable group of order $24$. 
Moreover, it is easily checked that $(y,a,x^2,x)$
is an induced $4$-path and hence $\mathcal{E}(K)$ is not a cograph. Hence by Lemma~\ref{l:Ma24_Lem1}, $\mathcal{E}(A_7)$ is not a cograph.
}
\end{example}

\begin{prop}\label{p:A7}  $\mathcal{E}(A_7)$ is a chordal graph which is not a cograph.
\end{prop}

\begin{proof} As shown in Example \ref{r:C_3:D_8}, $\mathcal{E}(A_7)$ is not a cograph.
 By Lemma~\ref{l:cut-cyc-simp}$(ii)$, to prove that $\mathcal{E}(A_7)$ is chordal, it is sufficient to prove that $\Gamma\coloneq \mathcal{E}(A_7)- U$ is chordal, for a certain subset $U\subseteq \mathrm{sl}(\mathcal{E}(A_7))\cup \mathrm{Cyc}(A_7).$ 
 By Corollary~\ref{c:estremiG}, $\mathrm{sl}(\mathcal{E}(A_7))$ contains the set $M(A_7)$  of elements of $G$ that generate a maximal cyclic subgroup: these are all the elements of orders $4,5,6$ or $7$ together with elements which are a product of two disjoint $3$-cycles. Also $1\in \mathrm{Cyc}(A_7)$. We take $U\coloneq M(A_7)\cup\{1\}$ and prove that $\Gamma\coloneq \mathcal{E}(A_7)- U$  is chordal. Note that $V(\Gamma)=A_7\setminus U$ consists only of the $3$-cycles, and the elements which are products of two disjoint $2$-cycles. 
 
 Arguing by contradiction, suppose that $n\geq4$ and that $\mathcal{C}=(x_1,x_2,\dots, x_n,x_1)$ is an induced $n$-cycle of $\Gamma$. Then $\mathcal{C}$ is also an induced cycle of $\mathcal{E}(A_7)$ and we may apply Lemma~\ref{l:general}. By Lemma~\ref{l:general}\,$(vii)$, adjacent vertices in $\mathcal{C}$ have different orders, and hence $n$ is even and we may assume that $|x_i|=2$ for $i\in[n]$ odd and $|x_i|=3$ for $i\in[n]$ even. Since conjugating by elements of $A_7$ sends induced cycles to induced cycles, we may assume further that $x_1=(12)(34)$. As $n\geq 4$,  $x_n$ and $x_2$ are distinct $3$-cycles and both must commute with $x_1$. Thus they belong to $\{(567),(576)\}$ and therefore $\langle x_2\rangle=\langle x_n\rangle=\langle (567)\rangle$, contradicting Lemma~\ref{l:general}\,$(ii)$.  Thus $\mathcal{E}(A_7)$ is chordal.
 \end{proof} 

\begin{corollary}\label{supersolvable} There exists a supersolvable  group whose enhanced power graph is chordal but not a cograph.
\end{corollary}

 \begin{proof} Consider the subgroup $K$ of $A_7$ defined in Example \ref{r:C_3:D_8}. We showed that  $\mathcal{E}(K)$ is not a cograph. On the other hand, by Proposition \ref{p:A7} and Corollary \ref{c:Ma24_Lem1}, $\mathcal{E}(K)$ is chordal.
 \end{proof}

Next we show that the result of Cameron~\cite[Proposition 8.7]{GraphsOnGroups} about the groups $\PSL_2(q)$ follows from our work in Section~\ref{sect:block}, giving a completely different proof.

\begin{lemma}\label{l:PSL2}
Let $G=\PSL_2(q)$ with $q\geq4$. Then $\mathcal{E}(G)$ is a block graph and a cograph, and the intersection of any two distinct maximal cyclic subgroups is trivial. 
\end{lemma}

\begin{proof} By \cite[II.8.5 Satz]{hup} the group $G=\PSL_2(q)$ admits a nontrivial partition consisting of cyclic subgroups. 
Thus the result follows from Theorem \ref{p:partition}.
\end{proof}

%
%     Ma Prop 1 & generalisation
%

To deal with the Suzuki groups, which also admit a nontrivial partition, we use Corollary \ref{c:simil_Ma}.

\begin{lemma}\label{l:suzuki}
Let $G=\Sz(q)$, with $q=2^{2m+1}$ for some $m\geq1$. Then $\mathcal{E}(G)$ is a cograph, and the intersection of any two distinct maximal cyclic subgroups has order at most $2$.
\end{lemma}

\begin{proof} We refer to \cite[Chapter XI, Section 3]{hup-black} for information about $G =\Sz(q)$. In particular $G<\SL_4(q)$.
Let  $r=2^{m+1}$.  Then  
$|G|=q^2(q-1)(q^2+1)$ and $(q+r+1)(q-r+1)=q^2+1$. Let $U$ 
be a Sylow $2$-subgroup of $G$. Then $U$ has order $q^2$ and exponent $4.$ Let $H$ be the subgroup of 
diagonal matrices in $G$. Then $H$ is isomorphic to 
the multiplicative group of the field $\F_q$ and hence 
$H\cong C_{q-1}.$
By \cite[XI, Theorem 3.10]{hup-black}, $G$ admits the 
partition $\mathcal{P}=\{U^g, H^g, T_1^g, T_2^g:g\in G\},$ 
where $T_1$ is a cyclic maximal torus of order $q+r+1$ and 
$T_2$ is a cyclic maximal torus of order $q-r+1$. As a 
consequence the orders of the elements in $G$ are the divisors of the numbers $4, q-1,q+r+1, q-r+1.$ Note that the orders of $H, T_1,T_2$ are odd and pairwise coprime. Moreover, since $H, 
T_1,T_2$ lie in the partition $\mathcal{P}$, 
these groups are maximal cyclic subgroups, by Remark~\ref{r:partn}$(i)$.
Let $\langle x\rangle$ be an arbitrary maximal cyclic subgroup of 
$G$. If $|x|$ is even, then as $x$ lies in some subgroup of $\mathcal{P}$, we must have   $x\in  U^g$ for some $g\in G$, so $|x|\mid 4$. On the other hand if  
$|x|$ is odd, then $x$ belongs to a conjugate of $H, T_1$, or $T_2$. Since these subgroups are cyclic, and 
$\langle x\rangle$ is maximal cyclic, it follows that 
$\langle x\rangle$ coincides with one of these conjugates. 
Thus the maximal cyclic subgroups of $G$
are given by the conjugates of $H, T_1, T_2$ together with the 
maximal cyclic subgroups of the Sylow $2$-subgroups of $G$.
As a consequence the intersection of two distinct maximal 
cyclic subgroups of $G$ has order at most $2$. By
Corollary \ref{c:simil_Ma}, it follows that $\mathcal{E}(G)$ is 
a cograph.
\end{proof}

We next consider three more infinite families of small rank Lie type simple groups, namely $\PSL_3(q)$ in Lemma~\ref{l:PSL(3,q)_cograph}, $\PSU_3(q)$ in Lemma~\ref{l:PSU(3,q)_cograph}, and ${\rm Ree}(q)$ in Lemma~\ref{l:Ree(q)_cograph}.

\begin{lemma}\label{l:PSL(3,q)_cograph}
Let $G=\PSL_3(q)$ for a prime power $q\geq2$. Then 
\begin{enumerate}
    \item[$(a)$]  $\mathcal{E}(G)$ is a cograph if and only if  $q=2$ or $q=4$. Moreover, for $q\in\{2,4\}$, distinct maximal cyclic subgroups of $\PSL_3(q)$ intersect in a $2$-group.
    \item[$(b)$] If $q\notin\{ 2,4\}$ and $\mathcal{E}(G)$ is $C_4$-free, then $q$ is a prime and $(q-1)/\gcd(3,q-1)$ is a prime power.
\end{enumerate}
\end{lemma}

\begin{proof}
$(a)$ If  $q\in\{2, 4\}$,  the orders of the nontrivial elements
of $G$  are prime numbers, or four (if $q=4$), see \cite[pp. 3, 23]{Atlas}, and hence  distinct maximal cyclic subgroups of $G$ intersect in a $2$-group. This also implies that $G$ is an EPPO group,  and so by Proposition \ref{known},  $\mathcal{E}(G)$ is a cograph. That $\mathcal{E}(G)$ is a cograph when $q=2$ also follows from Lemma \ref{l:PSL2} since $\PSL_3(2)\cong \PSL_2(7)$. Assume now that $q\notin\{ 2,4\}.$ 
Suppose first that $\gcd(3,q-1)=1$, so that $G=\PSL(3,q)=\SL(3,q)$. 
We show that $\mathcal{E}(G)$ is not a cograph by applying Proposition \ref{p:W}. 
Let $w\in\mathbb{F}_q$ be a generator of $\mathbb{F}_q^*$, and note that $w^{-2}\neq w$  since $q\notin \{ 2,4\}$. Define $a, x\in G$ by 
\begin{equation}\label{e:a-x}
a\coloneq\begin{pmatrix}
    1 &0 &1\\
    0& 1& 0\\
    0&0&1
\end{pmatrix}\quad \mbox{and}\quad x\coloneq\mathrm{Diag}(w,w^{-2},w).
\end{equation}
Then $|x|=q-1$, $|a|=p$, and $a$ lies in the 
centre of the Sylow $p$-subgroup $P$ of $G$ consisting of the upper unitriangular matrices. Also, $x$ 
commutes with $a$, and therefore $A:=\seq{a,x}$ is a cyclic 
subgroup of $G$ of order $p(q-1)$. 

We claim that $A$ is a maximal cyclic subgroup of $G$. Straightforward computations show that 
\begin{equation}\label{e:Ca}
 C_G(a)=\set{\begin{pmatrix}
    \alpha &\beta &\gamma\\
    0& \alpha^{-2}& \delta\\
    0&0&\alpha
\end{pmatrix}\vert \alpha,\beta,\gamma,\delta\in\mathbb{F}_q, \alpha\neq 0
}=P:\seq{x}   
\end{equation}
and 
$$
C_G(x)=\set{\begin{pmatrix}
    \alpha & 0 &\beta \\
    0& d^{-1} & 0 \\
    \gamma & 0 & \delta
\end{pmatrix}\vert \ 
Y:=\begin{pmatrix}
    \alpha  &\beta \\
    \gamma &  \delta
\end{pmatrix}\in \GL_2(q), \ 
d=\mathrm{det}(Y) 
% (\begin{pmatrix}
%     \alpha  &\beta \\
%     \gamma &  \delta
% \end{pmatrix})
}\cong \GL_2(q).
$$
Any maximal cyclic subgroup of $G$ containing $A$ must lie in $C_G(A)$, namely
$$
C_G(A)=C_G(a)\cap C_G(x)=\set{\begin{pmatrix}
    \alpha &0 &\gamma\\
    0& \alpha^{-2}& 0\\
    0&0&\alpha
\end{pmatrix}\vert \alpha,\gamma\in\mathbb{F}_q, \alpha\neq 0
}=A$$
and thus $A$  is maximal 
cyclic, proving the claim. 
Now let 
\begin{equation}
    \label{e:g}
g\coloneq\begin{pmatrix}
    1 & 0 & 0\\
    0 & 1 & 1\\
    0 & 0 & 1
\end{pmatrix}\in P.    
\end{equation}
Then $|g|=p$ and, and it follows from \eqref{e:Ca} that $g\in C_G(a)$, so $g^{-1}ag=a$.
Now  
$$
b:= g^{-1}xg=\begin{pmatrix}
    w & 0      & 0\\
    0 & w^{-2} & w^{-2}-w\\
    0 &   0    & w
\end{pmatrix}
$$
and since $w^3\neq 1$, we have $b\not\in C_G(A) = A$. 
Thus $b\neq x$, and $b$ is  a second element of $G$ of order $q-1$, and 
% . In fact $\seq{x}\cap \seq{b}=1$.  {\color{Blue} C: Francesco please check this is true. We need a little more than $x\ne b$ below.} Also 
$b$ commutes with $a$ since both $g$ and $x$ do. Let $B:=\seq{a,b}$. Then $B=A^g$ is a maximal cyclic subgroup of $G$ of order $p(q-1)$ and $B\neq A$. 
Note that $A\cap B$ contains $a$ but does not contain $x$.
As we noted above, $C_G(x)\cong \GL_2(q)$ and under this isomorphism $\seq{x}$ corresponds to the centre of $\GL_2(q)$. Therefore $C_G(x)$ contains a maximal cyclic subgroup $C$ isomorphic to a Singer cycle subgroup of $\GL_2(q)$, of order $q^2-1$, and $C$ contains $\seq{x}$ as a proper subgroup. Note that $C$ is a maximal cyclic subgroup of $G$ as well, since any cyclic subgroup containing $C$ must centralise $x$ and hence lie in $C_G(x)$. 
Now $A\cap C$ contains $x$, and $A\cap C$ does not contain $a$, since otherwise these two maximal cyclic subgroups of $G$ would be equal, but $|A|\ne |C|$. Thus since $a\in A\cap B$, it follows that  $A\cap B\not\leq A\cap C$.  Moreover, $A\cap C\not\leq A\cap B$, as otherwise $A\cap B$ contains $\seq{x,a}=A$ and thus $A=B$, a contradiction. Thus condition $(b)$ of Proposition \ref{p:W} does not hold, and it follows from Proposition~\ref{p:W} that $\mathcal{E}(G)$ is not a cograph, completing the proof in this case.

Assume now that $\gcd(3,q-1)=3$ with $q\neq 4$. Let $\lambda=w^{(q-1)/3}$, where $w$ is the generator of $\mathbb{F}_q^*$ above, so $|\lambda|=3$.  Let $H=\SL_3(q)$ and $Z:=Z(H)=\seq{\lambda I_3}$, and set $\overline{H} :=H/Z=\PSL_3(q)=G$.   We use the same approach as before and we define the following elements of $\overline{H}$:
\[
\overline{a}:= aZ,\ \overline{x}:= xZ,\ \overline{g}:= gZ,\ \overline{b}:= \overline{x}^{\overline{g}} = x^gZ =  bZ,\ 
\]
with $a, x, g$ as in \eqref{e:a-x} and \eqref{e:g}, and $b=x^g$ as above.  
% $$\overline{x}=\begin{pmatrix}
%     w & 0      & 0\\
%     0 & w^{-2} & 0\\
%     0 &   0    & w
% \end{pmatrix}\cdot Z
% \quad \mathrm{ and } \quad 
% \overline{a}=\begin{pmatrix}
%     1 & 0 & 1\\
%     0 & 1 & 0\\
%     0 & 0 & 1
% \end{pmatrix}\cdot Z.$$
Then $|\overline{x}|=(q-1)/3$, 
$|\overline{a}|=p$ and 
$\overline{A} =\seq{\overline{x},\overline{a}}$ is a cyclic subgroup of $\overline{H}$ of order $p(q-1)/3$. 

For $y\in\{a,x\}$, it is clear that $\overline{C_H(y)}\subseteq C_{\overline{H}}(\overline{y})$, and we claim that equality holds.
Now $\overline{a}$ has order $p$. Suppose that $\overline{h}=hZ$ centralises $\overline{a}$. Then $\overline{a}=\overline{a}^{\overline{h}} = \overline{a^h}$, and hence $a^h=az$ for some $z\in Z$. Since $|a|=|a^h|=p$ and $z$ is central, it follows that $1=(a^h)^p = a^pz^p=z^p$, and since $|z|$ divides $|Z|=3\ne p$, it follows that  $z=1$, and hence $h\in C_H(a)$, so   $\overline{C_H(a)}= C_{\overline{H}}(\overline{a})$. Now we consider $y=x$. Since $x$, acting on $\F_q^3$, has an $\omega$-eigenspace of dimension $2$ and an $\omega^{-2}$-eigenspace of dimension $1$, these eigenspaces are both left invariant by $\overline{x}=xZ$, and hence also by  $C_{\overline{H}}(\overline{x})$. From the form of $C_H(x)$ above we see that  $\overline{C_H(x)}$ is the full stabiliser in $\overline{H}$ of this decomposition, and hence $\overline{C_H(x)}= C_{\overline{H}}(\overline{x})$.

% show that 
% $$C_{\overline{G}}(\overline{x})=\overline{C_G(x)}
% \quad \mathrm{ and } \quad 
% C_{\overline{G}}(\overline{a})=\overline{C_G(a)},
% $$
% Set also $\overline{g}=\begin{pmatrix}
%     1 & 0 & 0\\
%     0 & 1 & 1\\
%     0 & 0 & 1
% \end{pmatrix}\cdot Z$ and 
% $\overline{b}=
% \overline{x}^{\overline{g}}$. 

Arguing as in the previous case, $\overline{A}=A/Z$ is a maximal cyclic subgroup of $\overline{H}$. 
Similarly $\overline{B} :=\seq{\overline{a},\overline{b}} = B/Z$ is a maximal cyclic subgroup of $\overline{H}$ and $\overline{A}\cap \overline{B}$  contains $\overline{a}$  but does not contain 
$\overline{x}$. Finally, $C_{\overline{H}}(\overline{x}) = C_H(x)/Z\cong 
\GL(2,q)/Z_0$, where $|Z_0|=3$, and hence $C_{\overline{H}}(\overline{x})$ contains maximal cyclic subgroups of $\overline{H}$ of order $(q^2-1)/3$ which correspond to quotients of Singer cycles of $\GL(2,q)$
. One of these is $\overline{C} =C/Z$ with $C$ as above, and $\overline{C}$ contains $\overline{x}$;  therefore $\overline{A} \cap \overline{C}$  contains $\overline{x}$ and does not contain $\overline{a}$. Therefore 
$\overline{A}\cap \overline{B}\not\leq \overline{A}\cap \overline{C}$ and $\overline{A}\cap \overline{C}\not\leq \overline{A}\cap \overline{B}$. Thus as in the previous case, condition $(b)$ of Proposition \ref{p:W} fails and hence $\mathcal{E}(\overline{H})$ is not a cograph.
\smallskip

$(b)$ 
We first prove that if $q$ is not a prime then $\mathcal{E}(G)$ is not $C_4$-free. Let $d=\gcd(3,q-1)$, and let $q=p^f$ with $p$ a prime and $f\geq 2,$ so $q\geq 9$ (since $q\ne 4$). Let $w$ be a generator of $\mathbb{F}_q^*$. Throughout the proof, for $a\in \SL_3(q)$, we denote by $\overline{a}$ its projection in $G$. Consider the following elements of 
$\SL_3(q)$ (with $a$ as in \eqref{e:a-x}),
$$a_1 = a=\begin{pmatrix}
1 & 0 & 1\\
0 & 1 & 0\\
0 & 0 & 1
\end{pmatrix}
\quad \textrm{and} \quad  
a_2=\begin{pmatrix}
1 & 0 & w\\
0 & 1 & 0\\
0 & 0 & 1
\end{pmatrix}.$$ 
Let $A\coloneq\seq{\overline{a_1},\overline{a_2}}\leq G$ and note that $|\overline{a_1}|=|\overline{a_2}|=p$. 
Since $q\neq p,$ the field elements $w$ and $1$ are $\mathbb{F}_p$-linearly independent, and so $A \cong C_p^2$. Now consider the element 
$b_1=x\coloneq\mathrm{Diag}(w,w^{-2},w)$ from \eqref{e:a-x}. As we  showed in part $(a)$,  $\overline{b_1}$ has order $(q-1)/d >1$ and in particular $\gcd(|\overline{b_1}|,|\overline{a_i}|)=1$.
By \eqref{e:Ca}, $b_1\in C_G(a_1)$ and it is easily checked that also $b_1\in C_G(a_2)$. It follows that $\overline{b_1}$ commutes with $A$. Now let
$$
c=\begin{pmatrix}
1 & 0 & 0\\
0 & 1 & 1\\
0 & 0 & 1
\end{pmatrix}.
$$ 
Again $c\in C_G(a_1)$ by \eqref{e:Ca} and it is easy to check that also $c\in C_G(a_2)$, so $\overline{c}$ commutes with $A$. Next define 
\[
b_2:= c^{-1}b_1c = \begin{pmatrix}
w & 0 & 0\\
0 & w^{-2} & w^{-2}-w\\
0 & 0 & w
\end{pmatrix}
\quad \mbox{so we have}\quad \overline{b_2}\coloneq\overline{c}^{-1}
\overline{b_1}\overline{c}.
\]
%
% Now let 
% $$b_1=\begin{pmatrix}
% w & 0 & 0\\
% 0 & w^{-2} & 0\\
% 0 & 0 & w
% \end{pmatrix}.$$
%  Since $|w|=q-1\neq 3$, $\overline{b_1}$ is a nontrivial element of $G$ of order $(q-1)/d$.  In particular $\gcd(|\overline{b_1}|,|\overline{a_i}|)=1$. Note also that, for every $i\in[2],$ we have $[b_1,a_i]=1$ and therefore $\overline{b_1}$ commutes with $A$. 
% Let also
% $$c=\begin{pmatrix}
% 1 & 0 & 0\\
% 0 & 1 & 1\\
% 0 & 0 & 1
% \end{pmatrix}.$$ 
% Since $[c,a_i]=1$, if we set
% $$\overline{b_2}\coloneq\overline{c}^{-1}
% \overline{b_1}\overline{c},$$  
% we have $[\overline{b_2},\overline{a_i}]=1$ for all $i\in [2]$ and $|\overline{b_2}|=|\overline{b_1}|$, with $\overline{b_2}$  the projection of 
% $$b_2=\begin{pmatrix}
% w & 0 & 0\\
% 0 & w^{-2} & w^{-2}-w\\
% 0 & 0 & w
% \end{pmatrix}.$$ 
Since $q\geq 9$, the $b_2$-entry $w^{-2}-w\ne 0$, and so 
$\overline{b_2}\not\in\seq{\overline{b_1}}$. Thus for $B\coloneq\seq{\overline{b_1},\overline{b_2}}$, the subgroup $H=AB$ satisfies the conditions of Proposition \ref{p:AB}, and this implies that $\mathcal{E}(G)$ is not $C_4$-free. 

To prove the last assertion of part $(b)$, we assume that $(q-1)/\gcd(3,q-1)$ has distinct prime divisors $r, s$, and we prove that $\mathcal{E}(G)$ is not $C_4$-free. One of $r, s$, say $r$, is different from $3$. Note that if $s=3$, then $9$ divides $q-1$. Let  $a, b\in\mathbb{F}_q^*$ be such that $|a|=r$, and if 
$s\ne3$ then $|b|=s$, while if $s=3$ then $|b|=9$. Set 
\[
h\coloneq \Diag(ab,1,(ab)^{-1}), \ g:= \Diag(1,ab,(ab)^{-1}), \  \mbox{and}\ H:=\langle h,g\rangle \cong \langle h\rangle \times \langle g\rangle.
\]
Then $H\leq \SL_3(q)$, and 
\[
H\cap Z=\{ h^ig^i=\Diag((ab)^i, (ab)^i, (ab)^{-2i})\mid i\in[r|b|], (ab)^{3i}=1\}.
\]

Thus if $s\ne 3$ then $H\cap Z=1$, and hence $\PSL_3(q)$ has a subgroup $H/(H\cap Z)\cong H\cong C_{rs}\times C_{rs}\cong C_r^2\times C_s^2$. Similarly, if $s=3$ then $H\cap Z=\langle (hg)^{3r}\rangle  \cong C_3$ and $\PSL_3(q)$ has a subgroup $H/(H\cap Z)\cong  C_r^2\times C_3\times C_9$, which in turn has a subgroup $C_r^2\times C_3^2$. In either case, it follows from  Corollary~\ref{c:AB_application} that  $\mathcal{E}(G)$ is not $C_4$-free. 
\end{proof}

Next we consider the simple groups $\PSU_3(q)$. Since $\PSU_3(2)$ is soluble, we consider only the case $q\geq 3$.

\begin{lemma}\label{l:PSU(3,q)_cograph}
Let $G=\PSU_3(q)$ for a prime power $q\geq3$. Then $\mathcal{E}(G)$ is not a cograph.
%\begin{enumerate}    \item[$(a)$]  $\mathcal{E}(G)$ is not a cograph, and 
%    \item[$(b)$] if $\mathcal{E}(G)$ is $C_4$-free, then \C{$q$ is a prime}.  
    % $(q+1)/\gcd(3,q+1)$ is a prime power.
%\end{enumerate}
\end{lemma}

\begin{proof}
Our basic reference for the proof of this lemma is \cite[II 10.12 Satz]{hup}.  We write $q=p^f$, with $p$ a prime and $f\geq1$, and  set $\mathbb{F}=\mathbb{F}_{q^2}$. Let $w$ be a generator of $\mathbb{F}^*$, $H=\SU_3(q)$, $Z=Z(H)$, and for each subgroup $L\leq H$ let $\overline{L}= LZ/Z$. Thus $G=\overline{H}=\PSU_3(q)$. Without loss of generality we take $G, H$ %\D{why do we refer both to $G$ and $H$? Is it not enough to tell about $H$?}
to preserve the unitary form with Gram matrix 
$$
J=\begin{pmatrix}
    0 & 0& 1\\
    0 & 1 & 0\\
    1 & 0 & 0
\end{pmatrix}.
$$
So the space $\F^3$ has ordered basis % \D{is not a basis ordered by definition??}\C{not usually} 
$\mathcal{B}=(e,z,f)$ such that, with respect to $J$, the unitary form satisfies $(e,e)=(f,f)=(e,z)=(f,z)=0, (e,f)=(z,z)=1$. The basis $\mathcal{B}$ is called the \emph{standard} or \emph{unitary basis} in \cite[Proposition 2.3.2]{KL}. 
 Note that $d:=|Z|=\gcd(3,q+1)$. 
Let 
$$
\overline{a}=\begin{pmatrix}
    1 & 0 & \gamma\\
    0 & 1 & 0\\
    0 & 0 & 1
\end{pmatrix}Z$$
where $\gamma $ is a fixed element of $\mathbb{F}^*$ such that $\gamma+\gamma^q=0$.  
Then $|\overline{a}|=p$ and $\overline{a}$ lies in centre of the  Sylow $p$-subgroup $Q$ of $G$ given by 
$$
Q=\set{\begin{pmatrix}
    1 & \alpha & \beta\\
    0 & 1 & -\alpha^q\\
    0 & 0 & 1
\end{pmatrix}Z\ \vert\ \alpha,\beta\in \mathbb{F}, \beta^q+\beta+\alpha^{1+q}=0}.
$$
Now $|Q|=q^3$, $Q'=Z(Q)=\set{\begin{pmatrix}
    1 & 0 & \beta\\
    0 & 1 & 0\\
    0 & 0 & 1
\end{pmatrix}Z\ \vert\ \beta^q+\beta=0}$. Moreover, $N_{G}(Q)=Q:L$ where $L=\seq{\overline{y}}$ with 
$$
\overline{y}=\begin{pmatrix}
    w^{-q} & 0 & 0\\
    0 & w^{q-1} & 0\\
    0 & 0 & w
\end{pmatrix}Z
$$ 
so that $|\overline{y}|=(q^2-1)/d$. 
Then, by easy computations, one shows that  
$$
C_{G}(\overline{a})=Q:\seq{\overline{x}}
$$ 
where $\overline{x}\coloneq\overline{y}^{q-1}$ has order $(q+1)/d$. Note that $\overline{x}$ is the projection modulo $Z$ of the diagonal matrix $\Diag(\lambda,\lambda^{-2},\lambda)$ where $\lambda=w^{q-1}\in \mathbb{F}_{q^2}$.  Assume first that $q\neq 8$, so that $\overline{x}\neq 1$. Then an 
easy computation shows that 
$C_Q(\overline{x})=Z(Q)$ and, 
since $Z(Q)$ has exponent $p$, 
the cyclic subgroup  $A := \seq{\overline{a},\overline{x}}$ is a maximal cyclic subgroup of $G$. Note that $|A|=p(q+1)/d$. Now, for any fixed $\overline{g}\in Q\setminus N_{G}(\seq{\overline{x}})$, the subgroup   $B := A^{\overline{g}}$ is also a maximal cyclic subgroup of $G$.  Since $\seq{\overline{x}}\neq \seq{\overline{x}^{\overline{g}}}$ the cyclic subgroup $B$ is different from $A$. Note also that  $A\cap B$ contains $\overline{a}$ and not $\overline{x}$. Now $\seq{\overline{x}}<\seq{\overline{y}}$ and if $C$ is any maximal cyclic subgroup of $G$ containing $\overline{y}$ then $|C|$ is a multiple of $|\overline{y}|=(q^2-1)/d$ and hence $C\neq A$ and $C\neq B$. Moreover $A\cap C$ contains $\overline{x}$ and does not contain $\overline{a}$. Thus the condition in Proposition \ref{p:W}$(b)$ fails for $A, B , C$, and hence $\mathcal{E}(G)$ is not a cograph.\\
We deal now with the remaining case $G=\PSU_3(8)$. 
Looking at the Atlas \cite[p. 64]{Atlas} we see that $G$ has a unique conjugacy class of involutions, and admits  elements of orders $4$, $6$ and $21$, which are necessarily maximal because there are no elements in $G$ of order strictly divisible by $4$, $6$ or  $21$. Moreover, $G$ has two conjugacy classes of elements of order $3$, one having centraliser of order $81$ and the other of order $1512$. Let $c\in G$ with $|c|=21$ and $C=\seq{c}\cong C_{21}.$ Consider $y= c^7$ and observe that  $|y|=3$ and $7\mid |C_G(y)|$. By Atlas information, we deduce that $|C_G(y)|=1512$ and hence there exists an involution $x\in G$ centralising $y$. Define $A=\seq{x, y}\cong C_6$. By the fact that all the involutions are conjugate, 
$x$ must be the power of some $b\in G$ with $|b|=4$. Define $B=\seq{b}\cong C_4.$ The groups $A,B,C$ are cyclic maximal and obviously distinct. Moreover, we have $A\cap C=\seq{y}$ and $A\cap B=\seq{x}$. Thus, the condition in Proposition \ref{p:W}$(b)$ fails for $A, B , C$, and $\mathcal{E}(G)$ is not a cograph.

\end{proof}

The last family we treat separately is the family of Ree groups ${\rm Ree}(q)=\null^2G_2(q)$, where $q=3^{2m+1}$, and since  ${\rm Ree}(3)=\PSL_2(8):3$ is not simple, we take $q>3$.

\begin{lemma}\label{l:Ree(q)_cograph}
Let $G={\rm Ree}(q)=\null^2G_2(q)$ with $q=3^{2m+1}>3$. % $q=3^n$ and $n=2m+1>1$. 
Then $\mathcal{E}(G)$ is not a cograph.
%\D{Open: is it chordal for some values of $q$?}
\end{lemma}

\begin{proof}
Our basic reference for this case is \cite{LeNu}. Since $3^2\equiv 1\pmod{8}$, we have $q+1\equiv 4 \pmod{8}$, so $(q-1)/2$  and $(q+1)/4$ are both odd integers.
The following hold in $G$. 
\begin{enumerate}
    \item Sylow $2$-subgroups of $G$ are of the form $C_2^3$ and are self-centralising.
    \item Any pair of $2$-subgroups of the same order are conjugate in $G$.
    \item If $x$ has order $2$ then $C_G(x)=\seq{x}\times H$ with $H\cong \PSL_2(q)$.
    \item If $T=C_2^2<G$, then there exists a cyclic Hall subgroup $L$ of order $(q+1)/4$ and an element $y\in G$ of order $6$ such that $N_G(T)=T:(L:\seq{y})$ and $C_G(L)=T\times L$.
\end{enumerate}
 
Let $T=\seq{t_1,t_2}= C_2^2$, and $\seq{a}=C_{(q+1)/4}$ centralising $T$, as in (4). Then 
$A :=\seq{t_1,a}$ and $B :=\seq{t_2,a}$ are distinct cyclic subgroups of $G$ of order $(q+1)/2$. Since, by (4), $C_G(a)=T\times \seq{a}$ the subgroups 
$A$ and $B$ are maximal cyclic subgroups of $G$.
Note that $A\cap B=\seq{a}$ and does not contain $t_1$. Also, by (3), $C_G(t_1)=\seq{t_1}\times H$ with $H\cong\PSL_2(q)$. There exists an element $h\in H$ of odd order $(q-1)/2$  centralising $t_1$. Thus $C:=\seq{t_1,h}$ is a cyclic subgroup of order $q-1$, and $C$ is a maximal cyclic subgroup of $C_G(t_1)$ since $\seq{h}$ is a maximal cyclic subgroup of $H=\PSL_2(q)$. Hence $C$ is maximal cyclic in $G$ since every cyclic subgroup of $G$  containing $C$ must lie in $C_G(t_1)$. Note that 
$A\cap C$ contains $t_1$ and not $a$. Therefore 
$A\cap B\not\leq A\cap C$ and $A\cap C\not\leq A\cap B$. Thus the condition in Proposition \ref{p:W}$(b)$ fails for $A, B , C$, and hence $\mathcal{E}(G)$ is not a cograph.
\end{proof}

We are now ready to give the formal proof of Theorem B. For convenience we recall the subset
$$
\mathcal{L}_1=\{\PSL_2(q), \  {\rm Sz}(q),\ \PSL_3(4)\},
$$
where $q>3$ in the case of $\PSL_2(q)$, and $q=2^{2m+1}>2$ for ${\rm Sz}(q)$. We prove that $\mathcal{L}_1$ is the set of all finite nonabelian simple groups $G$ such that $\mathcal{E}(G)$ is a cograph. Recall also the subset
\begin{equation} \label{e:list2}
\mathcal{L}_1\cup\set{A_7, M_{11}, J_1, \PSL_3(q) , \PSU_3(q),  {\rm Ree}(q), 
{}^2F_4(2)'},
\end{equation}  
where $q$ is an odd prime  and $(q-1)/(\gcd(3,q-1)$ is a prime power in the case of $\PSL_3(q)$;  $q$ is a prime power and $q>2$ for $\PSU_3(q)$; 
and $q=3^{2m+1}>3$ for ${\rm Ree}(q)$. We prove that this set is a superset of the list $\mathcal{S}$ of finite nonabelian simple groups $G$ for which  $\mathcal{E}(G)$ is $C_4$-free.

\medskip\noindent
\emph{Proof of Theorem B.}\quad 
Let $G$ be a finite nonabelian simple group. Our proof strategy is the following: we prove that 
\begin{center}
    \emph{$\mathcal{E}(G)$ is a cograph if and only if $G\in \mathcal{L}_1$},
\end{center} 
that is,  the conditions  (1) and (2) of Theorem B are equivalent. It follows from Lemmas~\ref{l:PSL2}, \ref{l:suzuki} and~\ref{l:PSL(3,q)_cograph} that all the groups in the list 
$\mathcal{L}_1$ satisfy condition (3) of Theorem B, that is, distinct  maximal cyclic subgroups of $G$ intersect in a $2$-subgroup. The fact that this condition  (3) implies that $\mathcal{E}(G)$ is a cograph is a consequence of Corollary 
\ref{c:simil_Ma}. Thus we will have proved that the three conditions (1)--(3) of Theorem B are equivalent. 

For the final assertion we note that, by Theorem A, for each $G\in\mathcal{L}_1$ we will have proved that  $\mathcal{E}(G)$ is chordal and hence $C_4$-free, so that $G$ lies in the list  $\mathcal{S}$ of simple groups $G$ such that $\mathcal{E}(G)$ is $C_4$-free. Also by Proposition~\ref{p:A7}, $\mathcal{E}(A_7)$ is chordal so $A_7\in\mathcal{S}$.  In the course of our proof, for a simple group $G$, we sometimes conclude that $\mathcal{E}(G)$ is not a cograph by proving that  $\mathcal{E}(G)$ is not $C_4$-free and applying Theorem A. The groups for which we fail to determine that `$\mathcal{E}(G)$ is not $C_4$-free' are precisely the groups in \eqref{e:list2}, proving that the list \eqref{e:list2} is indeed a superset for $\mathcal{S}$. Also, in the course of the proof, we will show that $\mathcal{E}(G)$ is not $C_4$-free for $G=M_{11}$ and $G=J_1$. That establishes the first inclusion for $\mathcal{S}$. The final assertion of Theorem B will thereby be established.

We consider the simple groups $G$ according to the various families given by the %finite nonabelian simple group classification 
CFSG.  Our proof often uses the following consequence of  Lemma~\ref{l:Ma24_Lem1}: 

\begin{center}
    \emph{If $\mathcal{E}(H)$ is not $C_4$-free, for some subgroup $H<G$, then \\
    $\mathcal{E}(G)$ is not $C_4$-free and hence also $\mathcal{E}(G)$ not a cograph.}
\end{center} 
We will often use this implication without further reference.

% Lemma~\ref{l:Ma24_Lem1}: if $\mathcal{E}(H)$ is not a cograph, or is not $C_4$-free, for some subgroup $H<G$, then $\mathcal{E}(G)$ is not a cograph, or is not $C_4$-free, respectively. We will use this implication without further reference.

\medskip\noindent
\emph{Alternating Groups. \quad }
Here $G=A_n$ with $n\geq5$. By \cite[Proposition 4]{Ma24}, $\mathcal{E}(G)$ is a cograph if and only if $n\in\{5,6\}$, and we note that $A_5\cong \PSL_2(4)$ and 
$A_6\cong \PSL_2(9)$ are listed in $\mathcal{L}_1$. 
Moreover, by Proposition \ref{p:A7},  $\mathcal{E}(A_7)$ is chordal and hence $C_4$-free. For $n\geq 8$, $\mathcal{E}(A_n)$ is not $C_4$-free because  it contains the induced $4$-cycle
$$\mathcal{C}=((12)(34), (567), (13)(24), (568),(12)(34)).$$

\medskip\noindent
\emph{Classical groups. \quad}
Suppose now that $G$ is a nonabelian finite simple classical group over a field of order $q=p^f$, for $p$ a prime and $f\geq1$.
% , then for each subgroup $H$ of $G$, $\mathcal{E}(H)$
% also has this property.

\medskip\noindent
\emph{Case: $G=\PSL_n(q)$.\quad } If $n=2$ then $\mathcal{E}(G)$ is a cograph by Lemma \ref{l:PSL2} and if $n=3$ then $\mathcal{E}(G)$ is a cograph only when $q=2$ and when $q=4$ by Lemma \ref{l:PSL(3,q)_cograph}. Also by Lemma \ref{l:PSL(3,q)_cograph}, if $q\ne 2,4$ and $\mathcal{E}(G)$ is $C_4$-free, then $q$ is an odd prime and $(q-1)/\gcd(3,q-1)$ is a prime power -- and these groups are in the list in \eqref{e:list2}. Assume now that $n\geq 4$, and  view $G$ as the group $S/Z$  where $S=\SL_n(q)$ and $Z\cong C_{d}$ is the subgroup of scalars in $S$, where $d=\gcd(n,q-1)$. If $p=2$ then $S=\SL_n(q)$ has a subgroup $H=\SL_4(2)\cong A_8$ such that $H\cap Z=1$, so $G$ has a subgroup $A_8$. Thus $\mathcal{E}(G)$ is not a cograph and is not $C_4$-free, since $\mathcal{E}(A_8)$ is not $C_4$-free. Suppose now that $q$ is odd, and consider 
%The diagonal subgroup $D=\{\Diag(a,b,(ab)^{-1},1^{n-3})\mid a,b\in \mathbb{F}_{q}^*\}$ satisfies $DZ/Z\cong C_{q-1}^2$, being $n\geq4$. Therefore if $q-1$ is not a prime power, 
%by Corollary~\ref{c:AB_application}(a), $\mathcal{E}(G)$ is not chordal and thus neither a cograph. Let $q-1$ be  a prime power,which means that $q\in\PP^-$ 
%(by Lemma~\ref{l:q-1}(a)).
%On the other hand, if $n=3$ and $(3,q-1)=3$, then $DZ/Z$ contains $C_{(q-1)/3}^2$ so, by Corollary~\ref{c:AB_application}(a), $(q-1)/3$ is a prime power, that is, $q-1=3\cdot p^s$, for some prime $p$ and $s\geq0$. This implies by Lemma~\ref{l:q-1}(e) that $q\in\PP^{-,3}$. Hence if $n=3$ then part (b)(i) holds, so we may assume that $n\geq4$ with $q\in\PP^-$. 
the subgroup  $D=\{\mathrm{Diag}(a,b,I_{n-4})\mid a,b\in \SL_{2}(p)\}$ of $S$. Then $DZ/Z\cong \SL_{2}(p)\times \SL_{2}(p)$ if $n\geq 5$, or  $DZ/Z\cong \SL_{2}(p)\circ \SL_{2}(p)$ if $n=4$, and so, by Corollary~\ref{c:AB_application},  
$\mathcal{E}(G)$ is not a cograph and is not $C_4$-free.

%\Fr{In virtue of Proposition \ref{p:GL(2,q)} we just need to check $\PSL(3,q_0)$ $q_0=3\cdot 2^s+1$ a prime, and in poarticular the parabolics $q_0^2:(1/3)GL(2,q_0)$. } 

\medskip\noindent
\emph{Case: $G=\PSU_n(q)$.\quad } Here $n\geq 3$ and $(n,q)\ne (3,2)$. 
If $n=3$, then by Lemma \ref{l:PSU(3,q)_cograph}, $\mathcal{E}(G)$ is not a cograph, and these groups are in the list in \eqref{e:list2}.
%, and if  $\mathcal{E}(G)$ is $C_4$-free then \C{$(q+1)/\gcd(3,q+1)$ is a prime power - check what to write here}. We add these groups to the list in \eqref{e:list2}. 
Now assume that $n\geq 4$,  and view $G$ as the group $S/Z$  where $S=\SU_n(q)$ and $Z\cong C_{d}$ is the subgroup of scalars in $S$, where $d=\gcd(n,q+1)$. Note that $S<\SL_n(q^2)$. 
% and prove that then $\mathcal{E}(G)$ is not $C_4$-free. Let $q^2=p^f$ ($p$ a prime), 
Consider the subgroup $D=\{\mathrm{Diag}(a,b,I_{n-4})\mid a,b\in \SU_{2}(p)\}$. As $\SU_{2}(p)=\SL_{2}(p)$ we have, for $q$ odd, $DZ/Z\cong \SL_{2}(p)\times \SL_{2}(p)$ if $n\geq 5$ and  $DZ/Z\cong \SL_{2}(p)\circ \SL_{2}(p)$ if $n=4$; so by Corollary~\ref{c:AB_application}, $\mathcal{E}(G)$ is not $C_4$-free in these cases. Assume therefore that $q$ is even. If $q>2$ we consider the subgroup $H=\{\Diag(a,b,I_{n-4})\mid a,b\in \SU_{2}(q)\}$ of $S$. Here  $\SU_{2}(q)=\SL_{2}(q)\cong \PSL_2(q)$ and $HZ/Z\cong \PSL_2(q)\times\PSL_2(q)$,  
and by Corollary~\ref{c:AB_application}$(e)$, $\mathcal{E}(G)$ is not $C_4$-free. This leaves the case $q=2$. If $n\geq 6$ then, by \cite[pp. 46, 115]{Atlas}, $A_8<\Sp_6(2)<\PSU_6(2)\leq \PSU_n(2)$, which by \cite[Proposition 3]{Ma24} implies that $\mathcal{E}(G)$ is not $C_4$-free.
Hence $n\in\set{4, 5}$. Since $\PSU_5(2)=\SU_5(2)> \SU_4(2)=\PSU_4(2)$, it is sufficient to show that $\PSU_4(2)$ is not $C_4$-free. So assume that $G=\PSU_4(2)=\SU_4(2)$. We will apply Proposition \ref{p:AB}. Let $\mathbb{F}=\mathbb{F}_4$ be the field of order $4$, and let $w$ be a generator of $\mathbb{F}^\ast$. Take also the  the unitary form with Gram matrix $I_4$,  and consider the following unitary matrices:
$$a_1=\begin{pmatrix}
    1 & 0 & 0 & 0\\
    0 & 0 & 1 & 0\\
    0 & 1 & 0 & 0\\
    0 & 0 & 0 & 1
\end{pmatrix}, \,  
a_2=\begin{pmatrix}
    1 & 0 & 0 & 0\\
    0 & 0 & w & 0\\
    0 & w^{2} & 0 & 0\\
    0 & 0 & 0 & 1
\end{pmatrix}, $$
$$b_1=\begin{pmatrix}
    w & 0 & 0 & 0\\
    0 & 1 & 0 & 0\\
    0 & 0 & 1 & 0\\
    0 & 0 & 0 & w^2
\end{pmatrix}, \,  
b_2=\begin{pmatrix}
    w & 0 & 0 & 0\\
    0 & w & 0 & 0\\
    0 & 0 & w & 0\\
    0 & 0 & 0 & 1
\end{pmatrix}.$$
Then $|a_1|=|a_2|=2$ and $|b_1|=|b_2|=3$, $[a_i,b_j]=1$ for every $i,j\in\set{1,2}$ and $A=\seq{a_1,a_2}$, 
$B=\seq{b_1,b_2}$ are both not cyclic and centralise each other. Hence, by Proposition \ref{p:AB}, $\mathcal{E}(G)$ is not $C_4$-free.

\medskip\noindent
\emph{Case: $G=\PSp_n(q)'$.\quad }  Here $n\geq 4$, and  we may assume that $(n,q)\ne (4,2)$ since $\PSp_n(q)'\cong\PSL_2(9)\cong A_6$ has already been treated. Thus $G=\PSp_n(q)$, and we view $G$ as the group $S/Z$  where $S=\Sp_n(q)$ and $Z\cong C_d$, with $d=\gcd(2,q-1)$, is the subgroup of scalars in $S$. We will examine a subgroup $H=\Sp_2(q)\times \Sp_2(q)\cong \SL_2(q)\times \SL_2(q)$ of $S$, with $H$ acting on a nondegenerate $4$-subspace $U$ and fixing the orthogonal complement $U^\perp$ pointwise.  Suppose first that $q=2^f$ is even, so $Z=1$ and $G$ has a subgroup isomorphic to $H$. If $q>2$ then, by  Corollary~\ref{c:AB_application}$(e)$, $\mathcal{E}(G)$ is not  $C_4$-free. 
If $q=2$ then $n\geq 6$ and by \cite[page 187]{KL}, $A_8\leq \Omega^+_6(2) < \Sp_6(2)\leq G$, and again $\mathcal{E}(G)$ is not $C_4$-free as $\mathcal{E}(A_8)$ contains an induced $C_4$. Now assume that $q=p^f$ is odd,  so $Z\cong C_2$, and we note that $H$ has a subgroup $H_0=\Sp_2(p)\times \Sp_2(p)\cong \SL_2(p)\times \SL_2(p)$. If $n\geq 6$ then $Z\cap H_0=1$ and so $G$ has a subgroup  $\SL_2(p)\times \SL_2(p)$; while if $n=4$ then $Z<H_0$ and $G$ has a subgroup $H_0/Z\cong \SL_2(p)\circ \SL_2(p)$. Thus Corollary~\ref{c:AB_application}  implies that $\mathcal{E}(G)$ is not $C_4$-free.  

\medskip\noindent
\emph{Case: Orthogonal groups.\quad }  Here $G=\POm^\varepsilon_n(q)$, where $n\geq 7$, $\varepsilon=\pm$ if $n$ is even, and $\varepsilon=\circ$ with $q$ odd if $n$ is odd. We view $G$ as the group $S/Z$  where $S=\Omega_n^\varepsilon(q)$ and $Z$ is the subgroup of scalars in $S$. Note that $|Z|\leq 2$. Suppose first that $q=p^f$ is odd. It follows from \cite[Lemma 4.1.1$(ii)$]{KL} that $S$ contains a subgroup $H=\Omega_7^\circ(p)$ acting naturally on a nondegenerate $7$-subspace $U$ and fixing $U^\perp$ pointwise. Also, by \cite[pp.~185--187]{KL}, the action of $A_8$ on its deleted permutation module $\mathbb{F}_{p}^7$ yields an embedding of $A_8$ in $H$ and therefore $\mathcal{E}(G)$ is not $C_4$-free.
%\cite[Propositions 4 and 5]{Ma24}. 
Thus we may assume that $q=2^f$  and hence $\varepsilon=\pm$ with $n\geq 8$ even. We use a similar argument: it follows from \cite[Lemma 4.1.1$(ii)$]{KL} that $S$ contains a subgroup $H=\Omega_6^+(2)$ acting naturally on a nondegenerate $6$-subspace $U$ of plus type and fixing $U^\perp$ pointwise. Since $Z\cap H=1$ it follows that $G$ has a subgroup isomorphic to $H=\Omega_6^+(2)\cong A_8$, and we  conclude that $\mathcal{E}(G)$ is not $C_4$-free.

\medskip\noindent
\emph{Exceptional groups of Lie type. \quad }
The exceptional groups of Lie rank 1 are $G= \Sz(q)$ and $G={\rm Ree}(q)$. By Lemma \ref{l:suzuki}, $\mathcal{E}(\Sz(q))$ is a cograph for all $q>2$, and we note that these groups are in $\mathcal{L}_1$. On the other hand by Lemma \ref{l:Ree(q)_cograph},  $\mathcal{E}({\rm Ree}(q))$ is not a cograph, but we have not proved whether or not this graph is $C_4$-free, so we add the Ree groups to the list in \eqref{e:list2}.  In considering the other exceptional Lie type groups below, we use information about their maximal subgroups from \cite{LSS}.

\medskip\noindent
\emph{Case: exceptional groups of Lie rank $2$.\quad }  These are the groups ${}^3D_4(q)$, $G_2(q)$ and ${}^2F_4(q)'$. First let $G={}^3D_4(q)$ or $G=G_2(q)$. Then by \cite[Table 5.1]{LSS}, $G$ has a subgroup $\SL_2(q)\times \SL_2(q)$ or $\SL_2(q)\circ \SL_2(q)$, for $q$ even or odd, respectively, and so by Corollary~\ref{c:AB_application}, $\mathcal{E}(G)$ is not $C_4$-free. So suppose that $G={}^2F_4(q)'$. Then $q$ is even. If $q>2$ then  by \cite[Table 5.1]{LSS}, $G$ has a subgroup $H=\Sp_4(q)\cong \PSp_4(q)$, while if $q=2$ then $G$ has a subgroup $H=\PSL_3(3)$,  see \cite[p. 74]{Atlas}. In either case we showed above that $\mathcal{E}(H)$ is not $C_4$-free, so $\mathcal{E}(G)$ is not $C_4$-free. 
%\D{COMMENT: It remains open its being $C_4$-free.}

%\Fr{[Open: Is $\mathcal{E}({}^2F_4(2)')$ chordal?]}

       \medskip\noindent
    \emph{Case: exceptional groups of Lie rank $> 2$.\quad }  These are the groups $F_4(q)$, $E_6(q)$, ${}^2E_6(q)$, $E_7(q)$ and $E_8(q)$.   We show that, for each of these groups $G$, $\mathcal{E}(G)$ is not $C_4$-free by finding a subgroup $H<G$ such that $\mathcal{E}(H)$ is not $C_4$-free. If $G=F_4(q)$ then by \cite[Table 5.1]{LSS}, $G$ has a subgroup $H={}^3D_4(q)$ and we showed in the previous case that $\mathcal{E}(H)$ is not $C_4$-free. Next if $G=E_6(q)$ or ${}^2E_6(q)$, then by  \cite[Table 9 and 10]{Craven2023}, $G$ has a subgroup $H=F_4(q)$ and we just showed that $\mathcal{E}(H)$ is not $C_4$-free. Finally if $G=E_7(q)$ or $E_8(q)$ then,  by \cite[Table 5.1]{LSS}, $G$ has a subgroup $H=E_6(q)$ or ${}^2E_6(q)$ and again we know that $\mathcal{E}(H)$ is not $C_4$-free.

\medskip \newpage
\noindent
\emph{Sporadic simple groups.} \quad 

  \medskip\noindent
\emph{Case: $G=M_{11}$.  \quad } We first prove that $\mathcal{E}(G)$ is  $C_4$-free, and then that  $\mathcal{E}(G)$  is not a cograph.  Assume first that $\mathcal{E}(G)$ contains an induced  
$C_4$. Then by  Proposition \ref{p:AB}, $G$ contains 
elements $a_1, a_2, b_1, b_2$ of prime power order such that $\gcd(|a_i|,|b_j|)=1$ for all $i,j\in[2]$, and the subgroups $A=\seq{a_1,a_2}$ and $B=\seq{b_1,b_2}$ are both non-cyclic and  centralise each other. In particular none of the $a_i, b_j$ lies in $M(G)$ and so, by \cite[p. 18]{Atlas}, they do not have order $5, 8$ or $11$. Thus by \cite[p. 18]{Atlas}, the possible orders $|a_i|, |b_j|\in \{2,3,4\}$, so without loss of generality $|a_i|=2$ or $4$, and $|b_j|=3$, and we note that $|a_ib_j|=3|a_i|$. However $G$ has no elements of order $12$, so  $|a_i|=2$.  GAP computations show  that the element $b_1$ has centraliser $C_G(b_1)= \langle b_1\rangle\times S$ where $S\cong S_3$, and  in fact all elements of order 3 have centraliser of this form since $G$ contains a unique conjugacy class of such elements. Now all elements of order $2$ in $C_G(b_1)$ lie in $S$, and we conclude therefore that $A\subseteq S$. In fact, as $A$ is 
non-cyclic and generated by two involutions, we have $A=\langle a_1,a_2\rangle=S$. Further $B$ centralises $A$, and hence also $B$ centralises the derived subgroup $A'\cong C_3$, that is, $B\leq C_G(A')\cong A'\times S_3$. Since all elements of order $3$ in $C_G(A')$ lie in its unique Sylow $3$-subgroup we  have $B=C_3^2$, which contains $A'$. This is a contradiction since each $a_i$ inverts $A'$, but also centralises $B$. Hence $\mathcal{E}(G)$ is $C_4$-free.

Now we show that $\mathcal{E}(G)$ is not a cograph using Proposition~\ref{p:W}. As mentioned above, every element $a$ of order $3$ has centraliser $C_G(a)\cong \langle a\rangle\times S_3$. Thus there are three distinct cyclic subgroups of $C_G(a)$ of order $6$ that pairwise intersect in $\seq{a}\cong C_3$. Further, by \cite[p. 18]{Atlas}, $G$ contains no elements of order a proper multiple of $6$, and so these three subgroups are maximal cyclic subgroups of $G$. Let $A$ and $B$ be two of them.  Also, by \cite{Atlas}, $G$ has a unique conjugacy class of involutions and a unique conjugacy class of maximal cyclic subgroups of order $8$; one of the latter, say  $C$, is such that $A\cap C$ has order $2$.  It follows that $A\cap B\not\leq A\cap C$ and 
$A\cap C\not\leq A\cap B$, and hence,  by Proposition \ref{p:W}, $\mathcal{E}(G)$ is not a cograph. 

%  As mentioned above $C_G(b_1)=\langle b_1\rangle \times S_3$, and by our assumptions $A\leq C_G(b_1)$. This implies that $A=\langle a_1,a_2\rangle\cong S_3$. 
% Further $B$ centralises $A$, and hence also $B$ centralises the derived subgroup $A'\cong C_3$, that is, $B\leq C_G(A')\cong A'\times S_3$. Since all elements of order $3$ in $C_G(A')$ lie in its unique Sylow $3$-subgroup we  have $B=C_3^2$, which contains $A'$. This is a contradiction since each $a_i$ inverts $A'$, but also centralises $B$. 

 \medskip\noindent
 {\it Case: $G=M_{22}$.\quad  } 
 Computations using GAP show that $G$ has a subgroup isomorphic to $A_4\times C_2^2$ , contained in the maximal subgroup $2^4:S_5$. It now follows from Corollary  \ref{c:AB_application}$(b)$ that $\mathcal{E}(G)$ is not $C_4$-free, and is not a cograph. 

%We prove that 
%$\mathcal{E}(G)$ %is not $C_4$-free, and hence also is not a cograph.  by 
%showing that $G$ contains  a subgroup $H$ such that $\mathcal{E}(H)$ is not $C_4$-free.
%Let $b$ be an element of order $3$ of $G$. Then $C_G(b)\cong 
%C_3\times A_4$. Let $A\cong C_2\times C_2$ be the Sylow $2$-
%subgroup of $C_G(b)$. Then  calculations in GAP show that 
%$H= C_G(A)\cong C_2\times C_2\times A_4$. This 
%subgroup is not $C_4$-free by Corollary 
%\ref{c:AB_application}. %Thus $M_{22}$ is not $C_4$-free.

\medskip\noindent
{\it Case: $G=J_1$. \quad}
We first prove that $\mathcal{E}(G)$ is not a cograph. 
Using \cite[p. 36]{Atlas}, we see that $G$ has maximal cyclic subgroups of orders: $15, 10$ and $6$, and moreover $G$ has unique conjugacy classes of cyclic subgroups  of orders $3$ and $5$. For an element $b$ of order $5$, we have $C_G(b)=\langle b\rangle\times S_3$, so $G$ has maximal cyclic subgroups $A, B$ of orders $15, 10$ respectively,  such that  $A\cap B=\langle b\rangle =C_5$. If $a$ is an element of $A$ of order $3$, then $C_G(a)= C_3\times D_{10}$ and hence $a$ lies also in a maximal cyclic subgroup $C$ of order $6$ and we have $A\cap C=\langle a\rangle \not\leq A\cap B$. Also $A\cap B\not\leq A\cap C$, and hence  by Proposition \ref{p:W}, $\mathcal{E}(G)$ is not a cograph.

By Corollary \ref{c:AB_maximal}, to show that $\mathcal{E}(G)$  is $C_4$-free, it is sufficient to prove that  $\mathcal{E}(M)$ is $C_4$-free   for each maximal subgroup $M$ of $G$. 
According to \cite{Atlas} and \cite{Janko}, the isomorphism classes of maximal subgroups of $G$ are:
\begin{itemize}
    \item $\PSL_2(11)$
    \item $2^3:7:3$ 
    \item $2\times A_5$ 
    \item $S_3\times D_{10}$ 
    \item Frobenius groups of type either $ 19:6$, or  $11:10$, or $7:6$.
\end{itemize}
By Lemma~\ref{l:PSL2}, $\mathcal{E}(\PSL_2(11))$  is a cograph and hence is $C_4$-free, by Theorem A. Also, 
since the Frobenius groups $M$ above have a nontrivial partition by cyclic subgroups, it follows from Proposition~\ref{t:char-block} and Theorem \ref{p:partition} that $\mathcal{E}(M)$ is a cograph and so is $C_4$-free. 
For the three remaining groups $M$ we use detailed information from Janko's paper \cite{Janko}, and we show 
that $M$ does not contain a subgroup $H=AB$ as in 
Proposition \ref{p:AB}$(b)$. To do this we assume that $H=AB$ is such a subgroup of $M$ and find a contradiction. Thus $A=\seq{a_1,a_2}$ and $B=\seq{b_1,b_2}$ are both non-cyclic and  centralise each other, and the elements $a_1, a_2, b_1, b_2$ have  prime power order such that $\gcd(|a_i|,|b_j|)=1$ for all $i,j\in[2]$. In particular  $\seq{a_ib_j}=\seq{a_i}\times \seq{b_j}$, so  none of the $a_i, b_j$ lies in $M(G)$.

First we consider $M=2^3:7:3$. Here the unique Syow $2$-subgroup $L=C_2^3$ of $M$ is a Sylow $2$-subgroup of $G$ and $M=N_G(L)$. 
Since, by \cite[p. 36]{Atlas}, an element of order $7$ generates a maximal cyclic subgroup of $G$, none of the $a_i, b_j$ has order $7$ and hence their orders are $2$ or $3$. Thus we may assume that each $|a_i|=2$ and each $|b_j|=3$. Thus $A=\seq{a_1,a_2}\leq L=C_2^3$ which implies that $A=C_2^2$, and $B=\seq{b_1,b_2}\leq C_M(A)$.   Now $C_M(A)\leq C_M(a_1)$, and by \cite[p.36]{Atlas}, $C_M(a_1)= \seq{a_1}\times K$ with $K=A_5$. Thus $B\leq K$ and since $B$ is not cyclic, $B\cong A_4$ or $A_5$. Also $A\leq C_G(a_1)$ and so $A\cap K=C_2$, and this group should be centralised by $B$, which is a contradiction.   

Now let $M=2\times A_5$. Since $M$ contains elements of order $|a_ib_j|=|a_i|\cdot |b_j|$, and $M$ has no elements of 
order $15$, it follows without loss of generality that $|a_i|=2$ for each $i\in[2]$,  and $|b_i|\in\{3,5\}$ for each $i\in [2]$.   
Let $r=|b_i|$. Then, since 
elements of order $r$  in $A_5$ are self-centralising, we have  $C_M(b_i)=C_{2r}$  and this subgroup contains a unique element of order $2$, and hence cannot contain $A$. Thus we have a contradiction. 

Finally consider 
$M=S_3\times D_{10}$. Then, as $|a_i|, |b_j|$ are prime powers, these orders lie in $\{2,3,5\}$. We show that indeed $|a_i|, |b_j|$ must lie in $\{3,5\}$.  Assume, by contradiction,  that one of them is $2$, say $|a_1|=2$. Then $|b_j|\in\{3,5\}$ for all $j\in [2]$. 
Since $M$ has a unique Sylow 
$3$-subgroup and a unique Sylow $5$-subgroup, and since $B=\seq{b_1,b_2}$ is not cyclic, 
this implies that $\{|b_1|,|b_2|\} = \{3,5\}$, but then we have 
$B = C_3\times C_5=C_{15}$, which is cyclic, and this is a contradiction. Thus $|a_i|, |b_j|\in\{3,5\}$ for all $i,j\in[2]$. This however implies,  for each $i,j\in [2]$, that $\{|a_i|,|b_j|\} = \{3,5\}$  since $\gcd(|a_i|,|b_j|)=1$, and hence $A,B$ are cyclic, which is a contradiction.  We conclude that $\mathcal{E}(J_1)$ is $C_4$-free.

  \medskip\noindent
\emph{All other sporadic groups. \quad } We show that for $G$ a sporadic simple group different from $M_{11}, J_1$, the graph $\mathcal{E}(G)$ is not $C_4$-free, by exhibiting a
subgroup $H$ for which $\mathcal{E}(H)$ is not $C_4$-free. Often there are several choices for $H$, and we give just one in Table~\ref{tab:sporad}. The column of Table~\ref{tab:sporad} headed `Comments' refers to where we deduce that $\mathcal{E}(H)$ is not $C_4$-free.  
% and apply Corollary \ref{c:Ma24_Lem1}. Sometimes Corollary \ref{c:AB_application} % can be applied to show that $\mathcal{E}(H)$ is not chordal.} 
% \D{I do not understand this reference. Do you have Corollary \ref{c:AB} in mind?But I think that you are just finding some subgroups with enhanced power graph not a cograph, isn't it? I do not see $S_{10}$ as a group as in Corollary \ref{c:AB} 
% can be invoked to show that $\mathcal{E}(H)$ is not  $C_4$-free. 
The reference for the different choices of the subgroup $H$ is \cite{Atlas}. For example, we have  ${}^3D_4(2) < Th < B$, and $C_3^2\times C_2^2< 3^2:4\times A_6 < ON$.
\qed

  \begin{table}
   \caption{ Subgroups $H$ of sporadic groups $G$ for Theorem~B}
   \vskip 2mm
 \begin{tabular}{lll}
   \hline
 $H$  &  $G$ &  Comments\\
     \hline
  $A_4\times C_2^2$  &  $J_1$ & Corollary~\ref{c:AB_application}  \\
  $A_4\times S_3$  &  $M_{12}, J_2, Suz, He$ & Corollary~\ref{c:AB_application}  \\
 $A_8$  &  $M_{23}, M_{24}, HS, HN, Co_{1}, Co_{2}, Co_{3}, $ &   \\
        &  $Fi_{22}, Fi_{23}, Fi_{24}', B, J_{4}, Ru$ & Alternating case  \\
$G_2(5)$  &  $Ly$ &  Exceptional rank $2$ case  \\
 $C_3^2\times C_2^2$  &  $ON$  &   Corollary~\ref{c:AB_application}  \\
 $\PSU_4(3)$  &  $McL$ & Classical case $\PSU$ \\
 ${}^3D_4(2)$  &  $Th, M, B$  &  Exceptional rank $2$ case   \\
 \hline
 \end{tabular}
 \label{tab:sporad}
 \end{table}

% If $G=M_{12}$, take $H=A_4\times S_3$. \\
% If $G=M_{23}$ or $G=M_{24}$, then $H=A_8$.\\
% If $G= M^cL$, then $G$ contains $H=\PSU_4(3)$ and we have shown above that the enhanced power graph of $\PSU_4(3)$ is not $C_4$-free. Thus $\mathcal{E}(G)$ is not a $C_4$-free.\\
% If $G=J_2$, take $H=A_4\times A_5$.\\
% If $G=He$, take $H=S_4\times L_3(2)$. \\
% If $G=Ru$, take $H=A_8$. \\
% If $G=Suz$, take $H=A_6\times A_5$. \\ 
% If $G=O'N$, take $H=3^2:4\times A_6$. \\ 
% If $G=Co_3$, take $H=M_{23}$. \\ 
% If $G=Co_2$, take $H=M_{23}$. \\ 
% If $G=Fi_{22}$ take $H=S_{10}$. \\ 
% If $G=HN$, take $H=A_{12}$. \\ 
% If $G=Ly$, take $H=G_2(5)$. \\ 
% If $G=Th$, take $H=\null^3 D_4(2)$. \\ 
% If $G=Fi_{23},$ take $H=S_{12}$. \\
% If $G=Co_1$, take $H=Co_2$. \\
% If $G=J_4$, take $H=M_{24}$. \\
% If $G=Fi_{24}'$, take $H=Fi_{23}$. \\
% If $G=B$, take $H=Fi_{22}$. \\
% If $G=M$, take $H=Th$. 

\subsection{Some comments on power graphs}

We close this section with some comparisons with the power graph.
The finite nonabelian  simple groups for which the power graph is a cograph were determined in \cite[Theorem 1.3]{CA}, and consist of one individual group  $\PSL_3(4)$ together with the groups $\PSL_2(q)$ and $\Sz(q)$ with the prime power $q$ satisfying certain number theoretic conditions.  Thus, by Theorem B, the set of finite nonabelian simple groups $G$ such that $\mathcal{P}(G)$  is a cograph is a proper subset of the set of such groups $G$ for which $\mathcal{E}(G)$  is a cograph. At present it is not known whether there are infinitely many values of $q$ satisfying the conditions in \cite[Theorem 1.3]{CA} and hence whether there are infinitely many finite nonabelian simple groups  $G$ such that $\mathcal{P}(G)$  is a cograph \cite[Problem 1.4]{CA}. 

On the other hand Theorem B gives two infinite families of finite nonabelian  simple groups  $G$ such that $\mathcal{E}(G)$  is a cograph, and as we have discussed, the implication `$\mathcal{P}(G)$  a cograph $\Longrightarrow$ $\mathcal{E}(G)$ a cograph'  holds for finite simple groups. While there is no apparent reason for this implication to hold in general, other results in \cite{manna} and \cite{CA} seem to indicate that the following conjecture just might be true.
For instance, \cite[Theorem 12]{manna} states that the nilpotent groups whose power graph is a cograph are given only by $p$-groups or cylic groups $C_{pq}$ with $p, q$ distinct primes. Clearly this class is properly contained in the class of nilpotent groups whose enhanced power graph is a cograph (Theorem \ref{t:Ma24_Thm2}).
Another interesting example is given  by \cite[Theorem 4.1]{CA} which states that a  group $G$ in which the maximal cyclic subgroups intersect trivially has power graph being a cograph if and only if the orders of the maximal
cyclic subgroups are  prime powers or the product of two distinct primes. By Theorem \ref{p:partition} and  having in mind Remark \ref{r:partn}, we instead know that any group $G$ in which the maximal cyclic subgroups intersect trivially has enhanced power graph being a cograph. 
 
\begin{conjecture} \label{prob2} For every finite  group $G$, if the  power graph $\mathcal{P}(G)$  is a cograph, then also the enhanced power graph $\mathcal{E}(G)$  is a cograph.
\end{conjecture}
We emphasise that, other examples confirming the conjecture are given by \cite[Theorems 1.2, 4.2, 4.3, 4.4]{CA}.

\bigskip

\centerline{\bf Acknowledgement}

\medskip
We wish to thank José C\'aceres and Michel Habib for several useful discussions on the classes of graphs considered in the paper.
%We thank the anonymous referee for a careful reading of our paper. 

%\D{References to be checked again:  which are missing and which are not used.}				

		\vspace{10mm}
\noindent {\Large{\bf Conflict of interest}}
\vspace{2mm}

\noindent Declarations of conflict of interest in the manuscript: none.

\end{document}